\documentclass[a4paper,10pt,reqno]{amsart}

\usepackage{amsthm,amsmath,amsfonts,amssymb}
\usepackage{color}

\usepackage{todonotes}

\usepackage{hyperref}
\usepackage{pgf, tikz-cd}
\usepackage{tikzscale}

\usepackage{mathrsfs}
\usetikzlibrary{arrows}

\usepackage{subfigure}

\definecolor{grey}{rgb}{.7,.7,.7}

\newcommand{\aplim}{\mathop{\text{\rm ap-lim}}}

\newcommand{\dist}{\mathop{\text{\rm dist}}}

\newcommand{\e}{\varepsilon}
\newcommand{\N}{\mathbb{N}}
\newcommand{\de}{\partial}
\newcommand{\cJ}{{\mathcal J}}

\newcommand{\cJt}{\widetilde{\mathcal J}}
\renewcommand{\phi}{\varphi}

\newcommand{\spt}{\mathop{\mathrm{spt}}}
\newcommand{\difsim}{\Delta}
\renewcommand{\div}{\mathop{\mathrm{div}}}
\renewcommand{\epsilon}{\varepsilon}

\def\C{\mathcal{C}}
\def\R{\mathbb{R}}

\def\H{\mathcal{H}}
\def\grad{\nabla}
\renewcommand{\implies}{\Rightarrow}

\theoremstyle{plain}
\newtheorem{thm}{Theorem}[section]
\newtheorem{lem}[thm]{Lemma}
\newtheorem{prop}[thm]{Proposition}

\theoremstyle{definition}
\newtheorem{defin}[thm]{Definition}

\theoremstyle{remark}

\newtheorem{ex}[thm]{Example}

\title[The prescribed mean curvature equation in weakly regular domains]{
The prescribed mean curvature equation in weakly regular domains}

\author{Gian Paolo Leonardi}
\address{Dipartimento di Scienze Fisiche, Informatiche e Matematiche, Universit{\`a} degli Studi di Modena e Reggio Emilia, Via Campi 213/b, I-41125 Modena, ITALY}
\email{gianpaolo.leonardi@unimore.it}

\author{Giorgio Saracco}
\address{Department Mathematik, Universit\"at Erlangen-N\"urnberg, Cauerst. 11, D-91058 Erlangen, Germany}
\email{saracco@math.fau.de}

\thanks{G.P. Leonardi and G. Saracco have been supported by GNAMPA projects: \textit{Problemi isoperimetrici e teoria della misura in spazi metrici} (2015) and \textit{Variational problems and geometric measure theory in metric spaces} (2016). G. Saracco has been also supported by the DFG Grant n. GZ:PR 1687/1-1}



\begin{document}

\definecolor{eqeqeq}{rgb}{0.87,0.87,0.87}
\definecolor{ffffff}{rgb}{1.,1.,1.}
\definecolor{black}{rgb}{0.,0.,0.}
\definecolor{zzffzz}{rgb}{0.6,1.,0.6}
\definecolor{ttqqqq}{rgb}{0.2,0.,0.}
\definecolor{qqqqtt}{rgb}{0.,0.,0.2}
\definecolor{uuuuuu}{rgb}{0.26,0.26,0.26}

\begin{abstract}
We show that the characterization of existence and uniqueness up to vertical translations of solutions to the prescribed mean curvature equation, originally proved by Giusti in the smooth case, holds true for domains satisfying very mild regularity assumptions. Our results apply in particular to the non-parametric solutions of the capillary problem for perfectly wetting fluids in zero gravity. Among the essential tools used in the proofs, we mention a \textit{generalized Gauss-Green theorem} based on the construction of the weak normal trace of a vector field with bounded divergence, in the spirit of classical results due to Anzellotti, and a \textit{weak Young's law} for $(\Lambda,r_{0})$-minimizers of the perimeter. 
\end{abstract}

 \hspace{-3cm}
 {
 \begin{minipage}[t]{0.6\linewidth}
 \begin{scriptsize}
 \vspace{-2cm}
 This is a pre-print of an article published in \emph{NoDEA Nonlinear Differ. Equ. Appl.}. The final authenticated version is available online at: http://dx.doi.org/10.1007/s00030-018-0500-3
 \end{scriptsize}
\end{minipage} 
}

\maketitle

\section*{Introduction}

Let $\Omega$ be an open bounded set in $\R^{n}$ and let $H:\Omega\to \R$ be a Lipschitz continuous function. A classical solution to the Prescribed Mean Curvature equation is a function $u:\Omega\to \R$ of class $C^{2}$ satisfying 
\begin{equation}\label{PMC}\tag{PMC}
\div \left(\frac{\nabla u(x)}{\sqrt{1+|\nabla u(x)|^{2}}}\right) = H(x) \qquad  \forall\, x\in \Omega\,.
\end{equation}
The left-hand side of \eqref{PMC} corresponds to the mean curvature of the graph of $u$ at the point $(x,u(x))$. The existence and the properties of  solutions to \eqref{PMC}, possibly satisfying some given boundary conditions, have been the object of extensive studies in the past, also due to the close connection between \eqref{PMC} and capillarity. After the pioneering works by Young \cite{Young}, Laplace \cite{Laplace}, and Gauss \cite{Gauss}, it is nowadays a well-known fact that the mean curvature of a capillary surface in a cylindrical container with cross-section $\Omega$ is determined by the surface tension, by the wetting properties of the fluid with respect to the container, and by the presence of external forces such as gravity. The modern theory of capillarity has its roots in a series of fundamental papers by Finn \cite{Finn1965}, Concus-Finn \cite{Concus1969,Concus1974,Concus1974a}, Emmer \cite{Emmer1973,Emmer1976}, Gerhardt \cite{gerhardt1974existence,gerhardt1975capillarity,gerhardt1976global}, Giaquinta \cite{Giaquinta1974}, Giusti \cite{Giusti1976,Giusti1978}, and many others (see \cite{Finn1986} and the references therein). Other contributions to the theory have been obtained in various directions, see for instance Tam \cite{tam1986existence,tam1986regularity}, Finn \cite{Finn1988}, Concus-Finn 
\cite{Concus1991}, Caffarelli-Friedman \cite{caffarelli1985regularity}, as well as more recent works by De Philippis-Maggi \cite{DePhilippisMaggi2015}, Caffarelli-Mellet \cite{caffarelli2007capillary} and Lancaster \cite{Lancaster2010}. However the above list is far from being complete.

A necessary condition on the pair $(\Omega,H)$ for the existence of a solution to \eqref{PMC} can be easily found by integrating \eqref{PMC} on any relatively compact set $A\subset \Omega$ with smooth boundary. Indeed, by applying the divergence theorem we get
\[
\Big | \int_A H\ dx \Big | \le \int_{\de A} |\langle Tu,\nu\rangle|\, d\H^{n-1}\,,
\]
where $\nu$ is the exterior normal to $\de A$ and $\H^{n-1}$ is the Hausdorff $(n-1)$-dimensional measure in $\R^{n}$. Then using the fact that the vector field
\[
Tu(x) := \frac{\nabla u(x)}{\sqrt{1+|\nabla u(x)|^{2}}}
\] 
has modulus less than $1$ on $\Omega$, we obtain for every such $A$ the strict inequality 
\begin{equation}\label{eq:necessarycondition}
\Big | \int_A H\ dx \Big | < P(A),
\end{equation}
where $P(A)$ denotes the perimeter of $A$ (when $\de A$ is smooth, $P(A) = \H^{n-1}(\de A)$; more generally, $P(A)$ has to be understood in the sense of Definition \ref{def:per}). 

Notice that whenever $H$ is a non-negative constant on $\Omega$ one obtains the necessary condition $H< \frac{P(A)}{|A|}$ for all relatively compact subsets $A \subset \Omega$ with positive volume. Hence, the existence of solutions to \eqref{PMC} is closely related to the so-called \textit{Cheeger problem}, which consists in minimizing the quotient $\frac{P(A)}{|A|}$ among all $A\subset\subset \Omega$ (see for instance the review papers \cite{Leo} and \cite{Parini2011}, and references therein). 

In the fundamental paper \cite{Giusti1978}, Giusti proved that the necessary condition \eqref{eq:necessarycondition} is also sufficient for the existence of solutions to \eqref{PMC} in any bounded connected open set $\Omega$ with Lipschitz boundary. More specifically, he showed that if \eqref{eq:necessarycondition} holds together with the strict inequality
\begin{equation}\label{eq:nonext}
\left|\int_{\Omega}H\, dx\right| < P(\Omega)
\end{equation}
then one can find many variational solutions (see \cite{Giaquinta1974}) attaining any given Dirichlet $L^1(\de \Omega)$ boundary datum in a weak sense. On the other hand, a much more subtle situation occurs when the equality
\begin{equation}\label{eq:ext}
\left|\int_{\Omega}H\, dx\right| = P(\Omega)
\end{equation}
holds, as it corresponds to the so-called \textit{extremal case}. Whenever the pair $(\Omega, H)$ is such that both \eqref{eq:necessarycondition} and \eqref{eq:ext} are satisfied, we will call the pair \emph{extremal}.

Concerning the existence of solutions to \eqref{PMC} in the extremal case, one can essentially consider a suitably translated sequence of variational (non-extremal) solutions $u_{i}$ of \eqref{PMC}, defined on subsets $\Omega_{i}$ that converge to $\Omega$ both in volume and in perimeter, as $i\to\infty$. Then, one obtains a so-called \textit{generalized solution} $u$ defined on $\Omega$ as the limit of $u_{i}$ (in the sense of the $L^{1}$-convergence of the subgraphs, see \cite{Miranda1977}). The extremal case is particularly relevant because it corresponds to capillarity for a perfectly wetting fluid under zero-gravity conditions. By definition of perfect wetting, the fluid-gas interface meets the (smooth) boundary of the cylindrical container with a zero contact angle; in other words one expects that any solution $u$ in the extremal case automatically satisfies the boundary condition of Neumann type
\begin{equation}\label{eq:Neumann}
\langle Tu,\nu\rangle = 1\qquad \text{on }\de\Omega\,.
\end{equation}
At the same time, one also experimentally observes that the solution $u$ is unique up to additive constants.
This is what Giusti showed to be a consequence of a more general equivalence result (see Theorem 2.1 in \cite{Giusti1978}) that he proved for the extremal case under the strong regularity assumption $\de\Omega\in C^{2}$. Later, Finn observed that the regularity requirements on $\de\Omega$ can be reduced to piece-wise Lipschitz (see \cite[Chapter 6]{Finn1986}) if one is interested in the existence of solutions to \eqref{PMC} in the $2$-dimensional case, and to ``$C^{1}$ up to a $\H^{n-1}$-negligible set'' if uniqueness up to vertical translations has to be shown in the extremal case. However the question about the validity of Giusti's result under weaker assumptions on $\de\Omega$ is still not completely answered.

In this paper we prove Giusti's characterization of existence and uniqueness of solutions to \eqref{PMC} under very mild regularity hypotheses on $\Omega$, see Theorems \ref{thm:subextremalexistence}, \ref{thm:extremal}, and \ref{thm:ACE}. In particular, our results are valid for domains with inner cusps or with some porosity (see Example \ref{ex:porosita}), which of course fall outside of the Lipschitz class. 

Specifically, we assume that $\Omega\subset \R^{n}$ is an open, bounded set with finite perimeter, satisfying the following properties. First, we require that $\Omega$ coincides with its measure-theoretic interior (roughly speaking, we do not allow $\Omega$ to have ``measure-zero holes''). Then we assume the existence of $k=k(\Omega)>0$ such that
\begin{equation}\label{eq:key}
\min\{ P(E; \Omega^{\mathsf{c}}),P(\Omega \setminus E; \Omega^{\mathsf{c}}) \} \leq k\, P(E; \Omega)
\end{equation}
for all $E\subset \Omega$. Finally, we require that 
\begin{equation}\label{eq:weakmildreg}
P(\Omega) = \H^{n-1}(\partial \Omega).
\end{equation}
Whenever an open set $\Omega$ satisfies \eqref{eq:key} and \eqref{eq:weakmildreg}, we say that $\Omega$ is \textit{weakly regular}. 
We stress that weak regularity can be regarded as a minimal assumption in the following sense. On one hand, if one assumes \eqref{eq:weakmildreg}, then \eqref{eq:key} is equivalent to the existence of a continuous and surjective trace operator from $BV(\Omega)$ to $L^{1}(\de\Omega)$, by well-known results about traces of functions in $BV(\Omega)$ (see Theorem \ref{thm:Mazya2011}). On the other hand, by Federer's Structure Theorem (see Theorem \ref{thm:Fed}), \eqref{eq:weakmildreg} amounts to requiring that the set of points of $\de\Omega$ that are of density $0$ or $1$ for $\Omega$ is $\H^{n-1}$-negligible, which can be considered as a very mild regularity assumption on $\de\Omega$. Moreover, in the extremal case one can show that \eqref{eq:key} is automatically satisfied by $\Omega$, thus only \eqref{eq:weakmildreg} needs to be assumed (see \cite{Sar17}). 

The proofs of the above-mentioned theorems require some facts and preliminary results of independent interest.  

One of the key tools that we shall systematically use in our proofs is Theorem \ref{thm:interiorapprox} about the interior approximation of an open set $\Omega$ with finite perimeter satisfying \eqref{eq:weakmildreg}, by means of sequences of smooth sets that converge to $\Omega$ in measure and in perimeter. This result has been proved by Schmidt \cite{Schmidt2014}, here we only add to the statement the useful observation that, being $\Omega$ connected, one can find a sequence of connected smooth sets with the above-mentioned property. Another, more technical tool is the recent characterization of $W^{1,1}_{0}(\Omega)$ as the space of functions in $W^{1,1}(\Omega)$ having zero trace at $\de\Omega$, due to Swanson \cite{swanson2007} (see Theorem \ref{thm:W110}).

In Section \ref{sec:GaussGreen} we introduce some notions and prove some results, that will be needed in the following sections. Under the assumptions \eqref{eq:key} and \eqref{eq:weakmildreg}, we prove Theorem \ref{thm:GaussGreen} which states a generalized Gauss-Green formula valid for bounded continuous vector fields with bounded divergence and for $BV$ functions on $\Omega$. 
We recall that very general forms of the Gauss-Green Theorem have been already obtained by several authors, see for instance \cite{DeGiorgi61b, Federer45, Federer58}, \cite{BuragoMazja69, Volpert67a}, \cite{Anzellotti83, Ziemer83}, and \cite{Pfeffer05a, DePauwPfeffer04a}. We recall in particular the extensions of the divergence theorem for bounded, divergence-measure vector fields on sets with finite perimeter \cite{ChenFrid99a, ChenFrid03, ChenTorres05, ChenTorresZiemer09}. These last results rely on a notion of \textit{weak normal trace} of a bounded, divergence-measure vector field $\xi$ on the reduced boundary of $E$, where $E\subset\subset \Omega$ is a set of finite perimeter and $\Omega$ is the domain of the vector field, see \cite{AmbrosioCrippaManiglia2005, ComiPayne2018, CrastaDeCicco2017}. This notion of trace already appears in \cite{Anzellotti83}, in the special case of $E$ being an open bounded set with Lipschitz boundary. 
A crucial tool used in \cite{ChenTorresZiemer09} (see also \cite{ComiTorres2017}) is the approximation of $E$ by smooth sets which are ``mostly'' contained in the measure-theoretic interior of $E$ with respect to the measure $\mu = \div \xi$. Actually, this is the main reason why $E$ needs to be compactly contained in the domain of the vector field $\xi$. On the other hand, if such a domain $\Omega$ has finite perimeter and $P(\Omega) = \H^{n-1}(\de\Omega)$ then one can consider the vector field $\hat\xi$ defined as $\hat\xi = \xi$ on $\Omega$ and $\hat\xi = 0$ on $\R^{n}\setminus \Omega$, so that by relying on Theorem \ref{thm:interiorapprox} it is possible to show that $\div \hat\xi$ is a finite measure on $\R^{n}$. Then by applying \cite[Theorem 25.1]{ChenTorresZiemer09} one might show the validity of the divergence theorem for the field $\xi$ on $E=\Omega$, which in turn leads to the generalized Gauss-Green formula
\begin{equation}\label{GGctz}
\int_{\Omega} \phi\, \div \xi + \int_{\Omega}\nabla\phi \cdot \xi = \int_{\de\Omega} \phi\, [\xi\cdot \nu]\, d\H^{n-1}\,,
\end{equation}
where $\nu$ is the exterior weak normal to $\de^{*}\Omega$, $[\xi\cdot\nu]$ denotes the weak normal trace of $\xi$ on $\de^{*}\Omega$, and $\phi\in C^{\infty}_{c}(\R^{n})$. However, also in view of the results of Section \ref{section:extremality}, in Section \ref{subs:wnt} we give a very direct proof of \eqref{GGctz} when $\Omega\subset\R^{n}$ is an open bounded set satisfying \eqref{eq:key} and \eqref{eq:weakmildreg}. This will be accomplished by adapting the construction proposed by Anzellotti in \cite{Anzellotti83} (see also \cite{Beck-Schmidt2015,Scheven-Schmidt2016}). More precisely, we will show that \eqref{GGctz} holds for every bounded continuous vector field $\xi$ with divergence in $L^{\infty}(\Omega)$ and for any $\phi\in BV(\Omega)$. We remark that the extra assumptions on $\xi$ that we are requiring reflect the properties of the vector field $Tu$ when $u$ is a solution of \eqref{PMC} on $\Omega$. We also stress that all bounded and connected Lipschitz domains, as well as some domains with inner cusps or with some controlled porosity (see for instance Example \ref{ex:porosita}), are weakly regular and therefore \eqref{GGctz} holds on them. 

Finally, in the proof of Theorem \ref{thm:ACE} we shall use a so-called \textit{weak Young's law} for $(\Lambda,r_{0})$-minimizers of the perimeter, Theorem \ref{thm:lambdamin}, that was originally shown in \cite[Proposition 2.5]{LeoPra2015} in the special case of Cheeger sets.

Some final observations about the stability of the solution to \eqref{PMC} in the extremal case are made. On one hand it is well-known in capillarity theory that even small and smooth deformations of $\Omega$ typically produce discontinuous changes in the solution of the capillary problem in $\Omega$, and even the existence of such a solution in the non-parametric setting may instantaneously drop (see \cite{Finn1986}). 
On the other hand, in Proposition \ref{prop:stability} we give an answer to the question whether or not it is possible to obtain some stability result for the solution $u=u_{\Omega}$ of \eqref{PMC} when the pair $(\Omega, H)$ is extremal. Then, by coupling Proposition \ref{prop:stability} with the construction described in Example \ref{ex:porosita}, a sequence of non-smooth perturbations of a $2$-dimensional disk can be constructed, in such a way that the corresponding sequence of solutions to the capillary problem for perfectly wetting fluids in zero gravity converge (up to suitable translations, and in the sense of $L^{1}_{loc}$-convergence of the epigraphs) to the solution of the problem in the disk.

\section{Preliminaries}
\label{sec:relative inequalities}

We first introduce some basic notations. We fix $n\ge 2$ and denote by $\R^{n}$ the Euclidean $n$-space. Let $E\subset \R^{n}$, then we denote by $\chi_{E}$ the characteristic function of $E$. For any $x\in \R^{n}$ and $r>0$ we denote by $B_{r}(x)$ the Euclidean open ball of center $x$ and radius $r$. Given two sets $E,F$, we denote by $E\difsim F = (E\setminus F)\cup(F\setminus E)$ their symmetric difference. In order to define rescalings of sets, we conveniently introduce the notation $E_{x,r} = r^{-1}(E-x)$, where $E\subset \R^{n}$, $x\in\R^{n}$, and $r>0$. Let $E\subset \Omega\subset \R^{n}$ with $\Omega$ open; we write $E\subset\subset \Omega$ whenever the topological closure of $E$, $\overline E$, is a compact subset of $\Omega$. Given a Borel set $E$ we denote by $|E|$ its $n$-dimensional Lebesgue measure. Whenever a measurable function, or vector field, $f$ is defined on $\R^{n}$, we set $\|f\|_{\infty}$ for the $L^{\infty}$-norm of $f$ on $\R^{n}$.

\begin{defin}[Perimeter]\label{def:per}
Let $E$ be a Borel set in $\R^n$. We define the perimeter of $E$ in an open set $\Omega\subset \R^{n}$ as
\[
P(E; \Omega):=\sup \left\{ \int_\Omega \chi_E(x) \div g(x)\, dx\,: g\in C^1_c(\Omega;\, \R^n)\,, \|g\|_{\infty} \leq 1\right\}\,.
\]
We set $P(E) = P(E;\R^{n})$. If $P(E;\Omega)<\infty$ we say that $E$ is a set of finite perimeter in $\Omega$. In this case (see \cite{Maggi}) one has that the perimeter of $E$ coincides with the total variation $|D\chi_{E}|$ of the vector--valued Radon measure $D\chi_{E}$ (the distributional gradient of $\chi_{E}$), which is defined for all Borel subsets of $\Omega$ thanks to Riesz Theorem.
\end{defin}

\begin{defin}[Points of density $\alpha$]
Let $E$ be a Borel set in $\R^n$, $x\in \R^n$. If the limit
\[
\theta(E)(x) := \lim_{r\to0^+} \frac{|E\cap B_r(x)|}{\omega_nr^n}
\]
exists, it is called the density of $E$ at $x$. We define the set of points of density $\alpha\in [0,1]$ of $E$ as
\[
E^{(\alpha)} := \left\{ x\in \R^n\,:\, \theta(E)(x) = \alpha\right\}\,.
\]
We also define the essential boundary $\de^{e}E := \R^{n}\setminus (E^{(0)}\cup E^{(1)})$.
\end{defin}

\begin{defin}[Approximate limit]
Let $f$ be a measurable function or vector field defined on $\Omega$. Given $z\in \overline{\Omega}$ we write
\[
\aplim_{x\to z} f(x) = w
\]
if for every $\alpha>0$ the set $\{x\in \Omega:\ |f(x)-w|\ge \alpha\}$ 
has density $0$ at $z$.
\end{defin}

\begin{thm}[De Giorgi Structure Theorem]\label{thm:DeGiorgi}
Let $E$ be a set of finite perimeter and let $\de^*E$ be the reduced boundary of $E$ defined as
\[
\de^* E:=\left \{x\in \de^{e}E\,:\, \lim_{r\to 0^+} \frac{D\chi_E(B_r(x))}{| D\chi_E|(B_r(x))} = -\nu_E(x) \in \mathbb{S}^{n-1} \right\}\,.
\]
Then,
\begin{itemize}
\item[(i)] $\de^{*}E$ is countably $\H^{n-1}$-rectifiable in the sense of Federer~\cite{FedererBOOK};

\item[(ii)] for all $x\in \de^{*}E$, $\chi_{E_{x,r}} \to \chi_{H_{\nu_E(x)}}$ in $L^{1}_{loc}(\R^{n})$ as $r\to 0^{+}$, where $H_{\nu_{E}(x)}$ denotes the half-space through $0$ whose exterior normal is $\nu_{E}(x)$;

\item[(iii)] for any Borel set $A$, $P(E;A) = \H^{n-1}(A\cap \de^{*}E)$, thus in particular $P(E)=\H^{n-1}(\partial^* E)$;

\item[(iv)] $\int_{E}\div g = \int_{\de^{*}E} g\cdot \nu_{E}\, d\H^{n-1}$ for any $g\in C^{1}_{c}(\R^{n};\R^{n})$.
\end{itemize}
\end{thm}

\begin{thm}[Federer's Structure Theorem]\label{thm:Fed}
Let $E$ be a set of finite perimeter. Then, $\de^* E \subset E^{(1/2)} \subset \de^e E$ and one has
\[
\H^{n-1}\left(\de^e E \setminus \de^* E\right)=0 \,.
\]
\end{thm}

In what follows, $\Omega$ will always denote a \textit{domain} of $\R^{n}$, i.e., an open connected set coinciding with its measure-theoretic interior, that is, we assume that any point $x\in \R^{n}$, such that there exists a radius $r>0$ with the property $|B_{r}(x)\setminus \Omega|=0$, is necessarily contained in $\Omega$.

The next result combines \cite[Theorem 9.6.4]{Mazya2011} and \cite[Theorem 10 (a)]{AnzGia1978}.
\begin{thm}\label{thm:Mazya2011}
Let $\Omega\subset \R^{n}$ be a bounded domain with $P(\Omega) = \H^{n-1}(\de\Omega) <+\infty$. Then the following are equivalent:
\begin{itemize}
\item[(i)] there exists $k = k(\Omega)$ such that for all $E\subset \Omega$
\[
\min\{ P(E; \Omega^{\mathsf{c}}),P(\Omega \setminus E; \Omega^{\mathsf{c}}) \} \leq k P(E; \Omega);
\]

\item[(ii)] there exists a continuous trace operator from $BV(\Omega)$ to $L^{1}(\de\Omega)$ with the following property: for any $\phi\in L^{1}(\de\Omega)$ there exists $\Psi\in W^{1,1}(\R^{n})$ such that $\phi$ is the trace of $\Psi$ on $\de\Omega$.
\end{itemize}
\end{thm}

Another useful result is the characterization of $W^{1,1}_{0}(\Omega)$ as the space of functions in $W^{1,1}(\Omega)$ having zero trace at $\de\Omega$.
\begin{thm}[{\cite[Theorem 5.2]{swanson2007}}]\label{thm:W110}
Let $\Omega\subset \R^{n}$ be an open set and let $u\in W^{1,1}(\Omega)$. Then, $u\in W^{1,1}_{0}(\Omega)$ if and only if
\[
\lim_{r\to 0} \frac{1}{r^{n}}\int_{B_{r}(x)\cap \Omega}|u(y)|\, dy = 0
\]
for $\H^{n-1}$-almost all $x\in \de\Omega$.
\end{thm}

The following approximation theorem is essentially due to Schmidt \cite{Schmidt2014} and  will play a crucial role in the paper.
\begin{thm}[Interior smooth approximation]\label{thm:interiorapprox}
Suppose that $\Omega$ is a bounded open set in $\mathbb{R}^n$ such that $P(\Omega) = \H^{n-1}(\de\Omega) < +\infty$. Then, for every $\delta >0$ there exist an open set $\Omega_\delta$ with smooth boundary in $\mathbb{R}^n$ such that
\begin{equation}\label{proprSchmidt}
\Omega_\delta \subset\subset \Omega, \quad \quad \Omega \setminus \Omega_\delta \subset (\mathcal{N}_\delta (\partial \Omega) \cap \mathcal{N}_\delta (\partial \Omega_\delta)),\quad  \quad |\Omega\setminus\Omega_{\delta}|<\delta,\quad  \quad P(\Omega_\delta) \leq P(\Omega) +\delta,
\end{equation}
where $\mathcal{N}_\delta (A)$ denotes the $\delta$-tubular neighborhood of $A\subset \mathbb{R}^n$. Moreover, $\Omega_{\delta}$ can be chosen connected as soon as $\Omega$ is connected.
\end{thm}
\begin{proof}
The existence of $\Omega_{\delta}$ satisfying \eqref{proprSchmidt} is proved in \cite{Schmidt2014}. In order to show the last part of the statement, we fix a compact set $K\subset \Omega$ such that $|\Omega\setminus K|<\delta$, then setting $d = \min\{\dist(x,\de \Omega):\ x\in K\}$ we take a finite covering of $K$ by balls of radius $d/2$ and let $x_{1},\dots,x_{N}$ denote their centers. By connectedness, for any $h,k\in \{1,\dots,N\}$ there exists a path $\Gamma_{hk}\subset \Omega$ connecting $x_{h}$ to $x_{k}$, so that the set 
\[
\widetilde K = \bigcup_{h=1}^{N} B_{d/2}(x_{h}) \cup \bigcup_{h,k=1}^{N} \Gamma_{hk}
\]
is contained in $\Omega$, connected, compact, and such that $|\Omega\setminus \widetilde K|<\delta$. Let now $\tilde \delta =  \min(\min\{\dist(x,\de \Omega):\ x\in \widetilde K\},\delta)>0$, then by \eqref{proprSchmidt} with $\tilde \delta$ replacing $\delta$ we get an open set $\Omega_{\tilde \delta}$ which necessarily has a connected component $A$ containing $\widetilde K$, so that \eqref{proprSchmidt} and the last part of the statement are satisfied by setting $\Omega_{\delta} = A$.
\end{proof}

%\begin{defin}\label{def:innerminkowski}
%For any set $E\subset \R^n$ we define its $(n-1)$-inner Minkowski content $\MC(E)$ as
%\[
%\MC(E) := \lim_{\e\to 0^+} \frac{\left|E \setminus \{x\in E: \dist(x, \de E)\geq \e \}\right|}{\e}\,,
%\]
%whenever the limit exists.
%\end{defin}
%
%We remark that whenever the inner Minkowski content of $\Omega$ exists and coincides with $P(\Omega)$ a similar result to Theorem \ref{thm:interiorapprox} holds. Namely, one has the following proposition.
%\begin{prop}
%Let $\Omega$ be an open bounded set of finite perimeter, such that $P(\Omega) = \MC(\Omega)$. Then there exists a sequence $\{\Omega_{j}\}_{j}$ of relatively compact open sets with smooth boundary, satisfying the same properties as in Theorem \ref{thm:interiorapprox}.
%\end{prop}
%\begin{proof}
%It is enough to observe that $\MC(\Omega) = P(\Omega)$ coupled with coarea formula give
%\[
%\lim_{r\to 0^{+}}\frac{1}{r}\int_{0}^{r}\H^{n-1}(\de\Omega_{t})\, dt = 0\,,
%\]
%where $\Omega_{t} = \{x\in \Omega:\ \dist(x,\de\Omega)>t\}$. 
%Hence there exists a strictly decreasing, infinitesimal sequence $\{r_{j}\}_{j}$ such that the sequence $\{\Omega_{j}\}_{j}$ defined by $\Omega_{j} = \Omega_{r_{j}}$ satisfies the required properties, with a possible exception of the smoothness. Finally, in order to enforce $\de\Omega_{j}$ smooth one can apply standard approximation by smooth sets.
%\end{proof}

\section{Some technical tools}
\label{sec:GaussGreen}

We collect in this section some key notions and results that will be later needed. Our first aim is to prove the Gauss-Green Theorem \ref{thm:GaussGreen}, on which the main results of Section \ref{section:extremality} are based. For this we shall introduce the \textit{weak normal trace} of a vector field $\xi$ on $\de \Omega$, denoted as $[\xi\cdot \nu]$, as a suitable extension of the usual scalar product between the trace of $\xi$ and the normal to $\de\Omega$, whenever the former exists. It is indeed quite easy to prove that whenever the approximate limit of the vector field $\xi(x)$ exists as $x\to z\in \de^* \Omega$, then $[\xi\cdot \nu](z)$ equals the scalar product between that limit and the outer normal to $\de^{*}\Omega$ at $z$, see Proposition \ref{prop:classiclimit}. 

Our second tool is represented by a \textit{weak Young's law} for perimeter quasiminimizers, Theorem \ref{thm:lambdamin}, that will be needed in Section \ref{section:extremality} for the proof of the implication $(U) \implies (E)$ in Theorem \ref{thm:ACE}. A slightly less general form of this lemma has been proved in \cite{LeoPra2015}, in the context of Cheeger sets. Roughly speaking, it says that the inner boundary of any $(\Lambda,r_{0})$-minimizer of the perimeter in a domain $\overline\Omega$ must meet the reduced boundary of $\Omega$ in a tangential way.

\subsection{The Weak Normal Trace}\label{subs:wnt}

Let $\Omega\subset \R^{n}$ be open, bounded, and weakly regular, i.e., satisfying \eqref{eq:key} and \eqref{eq:weakmildreg}. We denote by $X(\Omega)$ the collection of vector fields $\xi\in L^{\infty}(\Omega;\R^{n})\cap C^{0}(\Omega;\R^{n})$ such that $\div \xi \in L^{\infty}(\Omega)$. Following Anzellotti \cite{Anzellotti83}, for every $u\in BV(\Omega)$ we define the pairing
\begin{equation}\label{pairing1}
\langle \xi,u\rangle_{\de\Omega} = \int_{\Omega} u\, \div \xi + \int_{\Omega}\xi\cdot Du\,.
\end{equation}
The map $\langle\cdot,\cdot\rangle_{\de\Omega}:X(\Omega)\times BV(\Omega)\to \R$ is bilinear. If $u,v\in W^{1,1}(\Omega)$ have the same trace on $\de\Omega$ then by Theorem \ref{thm:W110} there exists a sequence $\{g_{j}\}$ of functions in $C^{\infty}_{c}(\Omega)$ such that $g_{j}\to u-v$ weakly in $BV(\Omega)$, so that we have
\begin{align*}
\langle \xi,u - v\rangle_{\de\Omega} &= \int_{\Omega}(u-v)\div \xi + \int_{\Omega}\xi\cdot D(u-v) \\ 
&= \lim_{j} \int_{\Omega}g_{j}\div \xi + \int_{\Omega}\xi\cdot \nabla g_{j} = 0\,.
\end{align*}
This shows that the pairing defined in \eqref{pairing1} only depends on the trace of $u$ on $\de\Omega$. Then by Anzellotti-Giaquinta's approximation in $BV(\Omega)$ and by Theorem \ref{thm:Mazya2011} we infer that $\langle \xi,u\rangle_{\de\Omega} = \langle \xi,v\rangle_{\de\Omega}$ whenever $u,v\in BV(\Omega)$ have the same trace on $\de\Omega$. 

At this point we can show the continuity of the pairing \eqref{pairing1} in the topology of $L^{\infty}(\Omega;\R^{n})\times L^{1}(\de\Omega)$. The following, key lemma extends \cite[Lemma 5.5]{Anzellotti83}.
\begin{lem}\label{lem:AnzLem5.5gen}
Let $\Omega$ be weakly regular. Then for every $u\in L^{1}(\de\Omega)$ and $\e>0$ there exists $w_{\e}\in BV(\Omega)\cap C^{\infty}(\Omega)$ such that 
\begin{itemize}
\item[(i)]
the trace of $w_{\e}$ on $\de\Omega$ equals $u$ $\H^{n-1}$-almost everywhere on $\de\Omega$,

\item[(ii)] 
$\int_{\Omega} |\nabla w_{\e}| \le \int_{\de\Omega}|u|\, +\e$,

\item[(iii)] 
$w_{\e}(x) = 0$ whenever $\dist(x,\de\Omega)>\e$,

\item[(iv)] 
$\int_{\Omega}|w_{\e}| \le \e$,

\item[(v)] 
$\|w_{\e}\|_{L^{\infty}(\Omega)} \le \|u\|_{L^{\infty}(\de\Omega)}$.
\end{itemize} 
\end{lem}
\begin{proof}
Let us fix $\e>0$. By Theorem \ref{thm:Mazya2011} (ii) there exists $\Psi\in W^{1,1}(\R^{n})$ such that its trace on $\de\Omega$ coincides with $u$. Up to an application of Meyer-Serrin's approximation theorem, we can additionally assume that $\Psi\in C^{\infty}(\Omega)$. Moreover we fix a sequence $\{\Psi_{j}\}_{j}$ of smooth functions such that $\|\Psi - \Psi_{j}\|_{W^{1,1}(\R^{n})}\to 0$ as $j\to\infty$. Again by Theorem \ref{thm:Mazya2011} (ii) we have that the trace operator from $BV(\Omega)$ to $L^{1}(\de\Omega)$ is continuous, hence
\[
\int_{\de\Omega} |\Psi_{j}|\, d\H^{n-1} \to \int_{\de\Omega} |\Psi |\, d\H^{n-1} = \int_{\de\Omega} |u |\, d\H^{n-1}\qquad \text{as $j\to\infty$.}
\]
Given $\delta,\eta>0$ we define $\chi_{\delta,\eta}(x) = \chi_{\Omega_{\delta}} * \rho_{\eta}(x)$, where $\rho_{\eta}$ is a standard symmetric mollifier with support in $B_{\eta}(0)$, while 
$\Omega_{\delta}\subset\subset \Omega$ is obtained in virtue of Theorem \ref{thm:interiorapprox}, so that the Hausdorff distance between $\de\Omega_{\delta}$ and $\de\Omega$ is smaller than $\delta$ and $|P(\Omega_{\delta}) - P(\Omega)|\le \delta$. We note that up to choosing $\delta$ and $\eta$ small enough we get $\spt(\chi_{\delta,\eta})\subset\subset \Omega$, $\chi_{\delta,\eta} = 1$ on the set $\{x\in \Omega:\ \dist(x,\de\Omega)>\e\}$, and $\Big|\int_{\Omega}|\nabla \chi_{\delta,\eta}| - P(\Omega_{\delta})\Big| \le \delta$. 
Then we define $w_{\delta,\eta}(x) = \Psi(x)(1-\chi_{\delta,\eta}(x))$ and, for any fixed vector field $g\in C^{1}(\R^{n};\R^{n})$ with $\|g\|_{\infty}\le 1$ and compact support in $\Omega$, up to choosing $\delta$ and $\eta$ small enough as well as $j$ sufficiently large we obtain
\begin{align*}
\int_{\Omega} \nabla w_{\delta,\eta}\cdot g\, dx &= \int_{\Omega} (1-\chi_{\delta,\eta})\, \nabla\Psi\cdot g\, dx - \int_{\Omega} \Psi\, \nabla\chi_{\delta,\eta}\cdot g\, dx\\ 
&\le \int_{\Omega} (1-\chi_{\delta,\eta})\, |\nabla\Psi| - \int_{\Omega} \Psi_{j}\, \nabla\chi_{\delta,\eta}\cdot g\, dx - \int_{\Omega} (\Psi - \Psi_{j})\, \nabla\chi_{\delta,\eta}\cdot g\, dx\\ 
&\le \frac{\e}{4} + \int_{\Omega} |\Psi_{j}|\, |\nabla\chi_{\delta,\eta}|\, dx + \int_{\Omega} \chi_{\delta,\eta} \Big(\nabla(\Psi - \Psi_{j})\cdot g + (\Psi - \Psi_{j})\div g\Big)\, dx\\ 
&\le \frac{\e}{4} + \int_{\Omega} |\Psi_{j}|\, |\nabla\chi_{\delta,\eta}|\, dx + (1+\|\div g\|_{\infty})\int_{\Omega} \Big(|D(\Psi - \Psi_{j})| + |\Psi - \Psi_{j}|\Big)\, dx\\ 
&\le  \int_{\Omega} |\Psi_{j}|\, |\nabla\chi_{\delta,\eta}|\, dx + \frac{\e}{2} \le \int |\Psi_{j}|\, d|D\chi_{\Omega}| + \frac{3}{4}\e \le \int_{\de\Omega} |u|\, d\H^{n-1} + \e\,.
\end{align*}
We finally set $w_{\e} = w_{\delta,\eta}$ and, by taking the supremum over $g$, we find
\[
\int_{\Omega} |\nabla w_{\e}|\, dx \le \int_{\de\Omega} |u|\, d\H^{n-1} + \e\,,
\]
which proves (ii). Finally, (i), (iii) and (v) are immediate from the construction, while (iv) is easily shown to hold up to possibly taking smaller $\delta$ and $\eta$.
\end{proof}

Now, given $\e>0$ and $\phi\in BV(\Omega)\cap L^{\infty}(\Omega)$, taking $w_{\e}$ as in Lemma \ref{lem:AnzLem5.5gen} (with $u=\phi$ on $\de\Omega$), and setting $\Omega_{\e} = \{x\in \Omega:\dist(x,\de\Omega)\ge \e\}$ we obtain
\begin{align*}
|\langle \xi,\phi\rangle_{\de\Omega}| &= |\langle \xi,w_{\e}\rangle_{\de\Omega}| \\
&\le \|\phi\|_{L^{\infty}(\Omega)}\int_{\Omega\setminus\Omega_{\e}} |\div \xi|\ + \|\xi\|_{L^{\infty}(\Omega)} \int_{\Omega} |\nabla w_{\e}|\\
&\le \|\phi\|_{L^{\infty}(\Omega)}\,\int_{\Omega\setminus\Omega_{\e}} |\div \xi|\ + \|\xi\|_{L^{\infty}(\Omega)}\left(\int_{\de\Omega}|\phi|\, +\e\right)\,,
\end{align*}
which by the arbitrary choice of $\e$ leads to 
\begin{equation}\label{stimatraccianor}
|\langle \xi,\phi\rangle_{\de\Omega}|\le \|\xi\|_{L^{\infty}(\Omega)}\, \int_{\de\Omega}|\phi|\,.
\end{equation}
One can check by a truncation argument that \eqref{stimatraccianor} holds for each $\phi\in BV(\Omega)$. An immediate consequence of \eqref{stimatraccianor} is the fact that the linear functional $N_{\xi}:L^{1}(\de\Omega)\to \R$ defined as $N_{\xi}(u) = \langle \xi,u\rangle_{\de\Omega}$ is continuous on $L^{1}(\de\Omega)$, thus it can be represented by a function in $L^{\infty}(\de\Omega)$, hereafter denoted by $[\xi\cdot \nu]$. This function is the so-called \textit{weak normal trace} of the vector field $\xi\in X(\Omega)$ on $\de\Omega$. Another immediate consequence of \eqref{stimatraccianor} is the following $L^{\infty}$-estimate of the weak normal trace:
\begin{equation}\label{ntracestimaLinf}
\|[\xi\cdot\nu]\|_{L^{\infty}(\de\Omega)} \le \|\xi\|_{L^{\infty}(\Omega)}\,.
\end{equation}
Summing up, we have proved that \eqref{pairing1} can be rewritten in the form of the generalized Gauss-Green formula stated in the next theorem.
\begin{thm}\label{thm:GaussGreen}
Let $\Omega\subset \R^{n}$ be open, bounded and weakly regular. Let $\xi\in X(\Omega)$ and $\phi\in BV(\Omega)$, then
\begin{equation}\label{GGformula}
\int_{\Omega} \phi\, \div \xi\,  + \int_{\Omega} \xi\cdot D\phi = \int_{\de\Omega} \phi\, [\xi\cdot \nu]\, d\H^{n-1}\,.
\end{equation}
\end{thm}

The next proposition shows that the weak normal trace is a proper extension of the normal component of the usual trace of $\xi$ on $\de\Omega$, whenever such a trace exists in measure-theoretic sense. 
\begin{prop}\label{prop:classiclimit}
Let $\Omega\subset \R^{n}$ be open, bounded and weakly regular. Let $\xi\in X(\Omega)$ and let $z\in \de^{*}\Omega$ be a Lebesgue point for the weak normal trace $[\xi\cdot \nu]$. Assume 
\begin{equation}\label{eq:aplim}
\aplim_{x\to z} \xi(x) = w\,,
\end{equation} 
then
\begin{equation}\label{weaktracerep}
[\xi\cdot\nu](z) = w \cdot \nu(z)\,.
\end{equation}
\end{prop}
\begin{proof}
We can assume that $z=0$ up to a translation. We fix a sequence $r_{i}\downarrow 0$ as $i\to\infty$. Given any function (or vector field) $f$ defined in $\Omega$, we set 
\[
\Omega_{i} = r_{i}^{-1}\Omega,\qquad 
f_{i}(y) = f(r_{i}y)\,.
\]
We note that $D f_{i}(y) = r_{i} Df(r_{i}y)$ in the sense of distributions. By \eqref{eq:aplim} we infer that for all $\alpha>0$ the set 
\[
N_{i}(\alpha) = r_{i}^{-1}N(\alpha) = \{y\in \Omega_{i}:\ |\xi_{i}(y) - w|\ge \alpha\}
\]
satisfies
\begin{equation}\label{Nivazero}
\lim_{i\to\infty}|N_{i}(\alpha) \cap B_{1}| = 0\,.
\end{equation}
On the other hand, the fact that $z=0$ is by assumption a Lebesgue point for $[\xi\cdot \nu]$ implies that
\begin{equation}\label{tracelebesgue}
[\xi\cdot\nu](0) = \lim_{i\to\infty} \mu_{i}^{-1}\int_{\de\Omega_{i}\cap B_{1}} [\xi\cdot\nu]_{i}(y)\, d\H^{n-1}(y)\,,
\end{equation}
where $\mu_{i} = \H^{n-1}(\de\Omega_{i}\cap B_{1})$. Now we take $\delta\in (0,1)$ and set $\alpha = \delta^{2}$ and
\[
\phi(y) = \max(0,\min(1,(1-|y|)/\delta))\,.
\]
By Theorem \ref{thm:DeGiorgi}(ii), setting $H = H_{\nu(0)}$ for brevity, we obtain
\begin{equation}\label{phiOmegaH}
\left|\int_{\Omega_{i}\cap B_{1}}D\phi(x)\, dx - \int_{H\cap B_{1}}D\phi(x)\, dx \right| \le \delta^{-1}|(\Omega_{i}\difsim H)\cap B_{1}| = m_{i}(\delta)\to 0\qquad \text{as }i\to\infty\,.
\end{equation}
Moreover by Theorem \ref{thm:GaussGreen} we get for a suitable constant $C>0$
\begin{align}\notag
\left|\int_{H\cap B_{1}} D\phi(x)\, dx - \omega_{n-1}\nu(0)\right| &= \left|\int_{\de H\cap B_{1}}\phi(x)\, d\H^{n-1}(x)  - \omega_{n-1}\right|\\\notag
&= \omega_{n-1}\int_{0}^{1}[1-(1-\delta t)^{n-1}]\, dt
\\\label{phisudeH}
&\le C\delta\,.
\end{align}
Then by \eqref{tracelebesgue}, \eqref{phiOmegaH}, \eqref{phisudeH}, and Theorem \ref{thm:DeGiorgi}(ii), we find 
\begin{align}\nonumber
\omega_{n-1}\Big|[\xi\cdot\nu](0) - w\cdot \nu(0)\Big| &\le \left| \lim_{i\to\infty}\int_{\de\Omega_{i}\cap B_{1}}[\xi\cdot\nu]_{i}\, d\H^{n-1} - w\cdot \int_{H\cap B_{1}} D\phi(x)\, dx \right| + C\delta \\\nonumber 
&\le \limsup_{i\to\infty}\left|\int_{\Omega_{i}\cap B_{1}}\phi\, \div\xi_{i}\right| + \left|\int_{\Omega_{i}\cap B_{1}} (\xi_{i}-w)\cdot D\phi \right| + m_{i}(\delta) + C\delta\\\label{eq:errgauss}
&= \limsup_{i\to\infty}\big(A_{i} + B_{i}+m_{i}(\delta)\big) + C\delta\,.
\end{align}
Then we notice that $A_{i} + m_{i}(\delta)\to 0$ as $i\to\infty$, while
\begin{align*}
B_{i} &= 
\left|\int_{(\Omega_{i}\cap B_{1})\setminus N_{i}(\alpha)} (\xi_{i}-w)\cdot D\phi  + \int_{N_{i}(\alpha)\cap B_{1}} (\xi_{i}-w)\cdot D\phi\right|\\
&\le \frac{\omega_{n}\alpha}{\delta} + \frac{2\|\xi\|_{\infty}}{\delta} |N_{i}(\alpha)\cap B_{1}|\\
&\le \omega_{n}\delta + \frac{2\|\xi\|_{\infty}}{\delta} |N_{i}(\alpha)\cap B_{1}|\,.
\end{align*}
Therefore by passing to the limit as $i\to\infty$ in \eqref{eq:errgauss} and using \eqref{Nivazero} we finally get
\[
\omega_{n-1}\Big|[\xi\cdot\nu](0) - w\cdot \nu(0)\Big| \le  (\omega_{n}+C)\delta\,,
\]
which implies \eqref{weaktracerep} at once by the arbitrary choice of $\delta\in (0,1)$.
\end{proof}

In general, the weak normal trace $[\xi\cdot \nu]$ of a vector field $\xi\in X(\Omega)$ at $x\in \de\Omega$ does not coincide to any pointwise, almost-everywhere, or measure-theoretic limit of the scalar product $\xi(y)\cdot \nu(x)$, as $y\to x$. However, one should expect some weak-type convergence of the normal component of $\xi$ to the value of $[\xi\cdot \nu]$ at any Lebesgue point $x_0\in \de^*\Omega$. More precisely, let $\Omega_h$ be a sequence of relatively compact, open subsets of $\Omega$ with smooth boundary, that converge to $\Omega$ both in perimeter and volume (see Theorem \ref{thm:interiorapprox}). We can consider the corresponding sequence of Radon measures $\mu_h = \langle \xi, \nu_h\rangle \, \H^{n-1}\llcorner \de\Omega_h$. By Theorem \ref{thm:GaussGreen} one easily checks that $\mu_h$ weakly-$*$ converges to $\mu = [\xi\cdot\nu]\, \H^{n-1}\llcorner \de^* \Omega$ as $h\to\infty$. 

By a similar application of Theorem \ref{thm:GaussGreen} (simply take $\phi = \chi_{B_r(x_0)}$) one can more explicitly characterize the weak normal trace at $\H^{n-1}$-almost every point $x_0\in \de^*\Omega$ as the following limit of spherical averages, as pointed out for instance in \cite{ChenTorresZiemer09}:
\[
[\xi\cdot \nu](x_0) = \lim_{r\to 0} \frac{1}{\omega_{n-1}r^{n-1}} \int_{\de B_r(x_0)\cap \Omega} \xi(x) \cdot \frac{x-x_0}{|x-x_0|}\, d\H^{n-1}(x)\,.
\]
Nevertheless, such a characterization of the weak normal trace is not fully satisfactory, as one would expect to obtain coincidence with the classical trace in some special cases (see in particular the characterization of extremality discussed in Section \ref{section:extremality}). A more specific study of weak normal traces will appear in \cite{LeoSar_trace}.

\subsection{The weak Young's law for $(\Lambda,r_0)$-minimizers}
\label{section:appendix}

Let us start recalling the definition of $(\Lambda,r_0)$-minimizer of the perimeter.

\begin{defin}\label{def:LambdaMinimizer}
	Let $\Omega\subset \R^{n}$ be an open set of locally finite perimeter, and let $E$ be a measurable subset of $\Omega$. We say that $E$ is a $(\Lambda, r_0)$-perimeter minimizer in $\overline{\Omega}$ if there exist two constants $\Lambda \in [0, +\infty)$ and $r_0 >0$ such that for every $x\in \R^{n}$, every Borel set $F$ such that $F\Delta E$ is compactly contained in $B_r(x)\cap \overline{\Omega}$, and every $r< r_0$, one has
	\[
	P(E; B_r(x)) \leq P(F; B_r(x)) +\Lambda |F\Delta E|\,.
	\] 
\end{defin}

\begin{thm}[Weak Young's Law]\label{thm:lambdamin}
	Let $\Omega$ be an open set with locally finite perimeter and let $E$ be a $(\Lambda, r_0)$-minimizer in $\overline{\Omega}$. Then $\partial E \cap \Omega$ meets $\partial^* \Omega$ in a tangential way, i.e., for any $x\in \partial^* \Omega \cap \overline{(\partial E\cap \Omega)}$ one has that $x\in \partial^* E$ and $\nu_E (x) = \nu_\Omega (x)$.
\end{thm}

\begin{proof}
	Let us fix a point $x \in \partial^* \Omega \cap \partial E$ and let $x+H$ be the half space obtained by blowing up $\Omega$ around $x$. We divide the proof in three steps. In the first one we prove that $E$ and $\Omega$ have the same tangential space at $x$, while in the third one we prove that $x$ is in $\partial^* E$ and that the outward normal is equal to the one outward $\Omega$. Step 2 provides a tool to prove Step 3.\\
	\emph{Step 1.} Let us prove that $E$ has the same tangent space $x+H$ at $x$. In order to do so, we need to prove perimeter and volume density estimates for $E \subset \Omega$ at $x$.
	Fix $m(r) := |E \cap B_r(x)|$ so that one has $P(E;\partial B_r(x)) = 0$, $m'(r) = P(E\cap B_r(x), \partial B_r(x))$ and $m(r) > 0$ for almost every $r>0$. Being $E$ a $(\Lambda, r_0)$-minimizer, for any $r<r_0$ and any competitor $F$, such that $F\Delta E \subset\subset B_r(x)\cap \overline{\Omega}$, one obtains
	\[
	P(E; B_r(x)) \leq P(F; B_r(x)) + \Lambda|F\Delta E|.
	\]
	Fix two radii, $r_2 < r_1 < r_0$ and consider as competitor in $B_{r_1}(x)\cap \overline{\Omega}$ the set $F:= E\setminus B_{r_2}$. Therefore, exploiting the $\Lambda$-minimality one has
	\[
	P(E;B_{r_1}(x)) \leq P(F;B_{r_1}(x)) + \Lambda |E\Delta F| \leq P(E;B_{r_1}(x) \setminus B_{r_2}(x)) + m'(r_2) + \Lambda m(r_2).
	\]
	Thus
	\begin{equation}\label{eq:boundperimeter}
	P(E; B_{r_2}(x)) = P(E;B_{r_1}(x)) - P(E;B_{r_1}(x) \setminus B_{r_2}(x))\leq \Lambda m(r_2) +m'(r_2).
	\end{equation}
	Due to the latter and to the isoperimetric inequality, it follows
	\begin{align}
	c_1 m(r_2)^{\frac{n-1}{n}} &= c_1 |E\cap B_{r_2}(x)|^{\frac{n-1}{n}} \leq P(E\cap B_{r_2}(x)) \nonumber \\
	&= P(E; B_{r_2}(x)) + P(E\cap B_{r_2}(x); \partial B_{r_2}(x)) \leq \Lambda m(r_2) + 2m'(r_2).
	\end{align}
	Hence for $r_2$ small enough and for some uniform constant $c_2$ we have
	\[
	\frac{m'(r_2)}{m(r_2)^{\frac{n-1}{n}}} \geq c_2.
	\]
	By integrating this inequality on $(\rho/2, \rho)$ we obtain for $\rho$ small enough the volume density estimate
	\[
	m(\rho) \ge c_3\rho^n\,,
	\]
	where $c_3$ is a uniform constant.

	Regarding the perimeter, directly from ~(\ref{eq:boundperimeter}) one can infer that $P(E; B_{r_2}) \leq \Lambda \omega_n r_2^n +m'(r_2)$, which, for $r_2$ small enough implies
	\[
	P(E; B_{r_2}) \leq c_4 r_2^{n-1},
	\]
	which then yields the perimeter density estimate.
	
	Now blowing up $E$ at $x$ we find a limit set $E_\infty$ contained in the half-space $x+H$ with $x\in \partial E_\infty$. It can be shown that $E_\infty$ is not empty and minimizes the perimeter without volume constraint with respect to any compact variation contained in $x+H$. By convexity of $H$ and by a maximum principle argument \cite{Simon} one infers that $E$ admits $x+H$ as unique blow up at the point $x$.
	\\
	\emph{Step 2.} Let us prove that
	\begin{equation}\label{eq:step2}
	\lim_{r\to 0} \frac{P(E; B_r(x))}{r^{n-1}} = \omega_{n-1}
	\end{equation}
	holds. Let $E_r$ be $r^{-1}(E-x)$. Since the blow up of $E$ at $x$ is the half space $x+H$ one has the $L^1_{loc}$-convergence $\chi_{E_r} \to \chi_H$ as $r$ goes to $0$. By the lower semi-continuity of the perimeter we have
	\[
	\liminf_{r\to 0} \frac{P(E_; B_r(x))}{r^{n-1}} = \liminf_{r\to 0} P(E_r; B_1(0)) \geq P(H; B_1(0)) \geq \omega_{n-1},
	\]
	therefore to prove ~(\ref{eq:step2}) it is enough to show that
	\begin{equation}
	\limsup_{r\to 0} P(E_r; B_1(0)) \leq \omega_{n-1}.
	\end{equation}
	Argue by contradiction and suppose there exists a sequence of radii $r_i$ going to $0$ such that
	\begin{equation}
	P(E_{r_i}; B_1(0)) \geq \omega_{n-1} +\epsilon.
	\end{equation}
	Recall that $x \in \partial^* \Omega$, therefore for $r_i$ small enough one has
	\begin{equation}
	P(\Omega_{r_i}; B_s(0)) \leq s^{n-1}\omega_{n-1} +\epsilon/3, \qquad \text{for all $1<s<2$,}
	\end{equation}
	where $\Omega_{r_i}$ is defined in the same manner of $E_{r_i}$. Due to the $L^1$-convergence in $B_2(0)$ of $\chi_{E_{r_i}}$ to $\chi_H$ and by coarea formula one can find a suitable
	\[
	t\in \Bigg(1, \Big( \frac{\omega_{n-1} + \epsilon/2}{\omega_{n-1} + \epsilon/3}\Big)^{\frac{1}{n-1}}\Bigg)
	\]
	such that
	\begin{equation}
	P(\Omega_i; \partial B_t(0)) = P(E_i; \partial B_t(0)) = 0
	\end{equation}
	\begin{equation}
	\H^{n-1}(E_i\Delta \Omega_i \cap \partial B_t(0)) < \frac{\epsilon}{4}
	\end{equation}
	hold. Consider now the sets $F_i := (E\cup B_{tr_i}(x)) \cap \Omega$, for which, due to the previous, one has
	\[
	P(F_i, B_{r_0}(x)) = P(E; (\Omega \cap B_{r_0}(x))\setminus B_{tr_i}(x)) + P(\Omega; B_{tr_i}(x)) + r_i^{n-1}\H^{n-1}(E_i\Delta \Omega_i \cap \partial B_t(0)).
	\]
	For $r_i$ small enough that $tr_i < r_0$, the set $F_i$ is a competitor to $E$ in $B_{r_0}$, therefore
	\begin{align}
	r^{n-1}(\omega_{n-1} + \epsilon) &\leq P(E; B_{r_i}(x)) \leq P(E; B_{r_0}(x)) - P(E; (\Omega \cap B_{r_0}(x))\setminus B_{tr_i}(x))\nonumber \\ 
	&\le P(F; B_{r_0}(x)) -P(E; (\Omega \cap B_{r_0}(x))\setminus B_{tr_i}(x)) + \Lambda |F\Delta E|\nonumber \\ 
	&\le P(F; B_{r_0}(x))-P(E; (\Omega \cap B_{r_0}(x))\setminus B_{tr_i}(x)) + \Lambda|E\cap B_{tr_i}(x)| \nonumber \\ 
	& \le P(\Omega; B_{tr_i}(x)) + r_i^{n-1}\frac{\epsilon}{4} + \Lambda \omega_n(tr_i)^n \\ 
	&\le (tr_i)^{n-1}(\omega_{n-1} +\epsilon/3) + r_i^{n-1}\frac{\epsilon}{4} + \Lambda \omega_n(tr_i)^n \nonumber \\ 
	& < r_i^{n-1}(\omega_{n-1} +\epsilon/2) + r_i^{n-1}\frac{\epsilon}{2} \leq r^{n-1}(\omega_{n-1} + \epsilon), \nonumber
	\end{align}
	which leads to a contradiction.
	\\
	\emph{Step 3.} Owing to \eqref{eq:step2}, in order to show that $x\in \de^{*}E$ and that $\nu_E(x) = \nu_\Omega(x)$ it is enough to prove that
	\begin{equation}\label{entp}
	\lim_{r\to 0} \frac{D\chi_E(B_r(x)) \cdot v}{\omega_{n-1}r^{n-1}} = 1\,,
	\end{equation}
	where we have set $v = -\nu_\Omega(x)$. 
	In virtue of Theorem \ref{thm:DeGiorgi} (iv), for almost every $r>0$ one has
	\begin{align}
	D\chi_E(B_r(x))\cdot v &= \int_{E\cap \de B_r(x)} v\cdot N\, d\H^{n-1} = \int_{H\cap \de B_r(0)} v\cdot N\, d\H^{n-1} +A(x,r) \label{eq:prima}\\
	&= \omega_{n-1}r^{n-1} +A(x,r)\,, \nonumber
	\end{align}
	where $N$ is the outward normal to $\de B_r(x)$ and
	\begin{align*}
	|A(x,r)| &=\left|v\cdot \int_{\de B_r(x)} (\chi_E(y) - \chi_{x+H}(y))N(y)\, d\H^{n-1}(y) \right|\\
	&\leq \int_{\de B_r(x)} |\chi_E(y) - \chi_{x+H}(y)| \, d\H^{n-1}(y)\,.
	\end{align*}
	Now for any fixed $\delta >0$, define the set $\Sigma(x, \delta) \subseteq (0, +\infty)$ of radii $r>0$ such that $A(x,r)>\delta r^{n-1}$. Hence, by the $L^1_{\text{loc}}$-convergence of $r^{-1}(E-x)$ to the half-space $H$ we infer that
	\[
	\lim_{\rho\to 0^+} \frac{\H^1(\Sigma(x,\delta)\cap (0, \rho))}{\rho} = 0\,.
	\]
	Therefore, for any decreasing infinitesimal sequence of radii $\{r_i\}_i$ we can find another sequence $\{\rho_i\}_i$ such that $\rho_i \notin \Sigma(x,\delta)$ for all $i$ and $\rho_i = r_i +o(r_i)$ as $i\to \infty$. Suppose by contradiction that \eqref{entp} does not hold. Then, there exist $\alpha >0$ and a decreasing infinitesimal sequence $\{r_i\}_i$ such that 
	\begin{equation}\label{eq:contr}
	\left|	\frac{D\chi_E(B_{r_i}(x)) \cdot v}{\omega_{n-1}r_i^{n-1}}\right| \geq \alpha\,,
	\end{equation}
	for all $i\in \N$. By suitably choosing $\delta$ as $\omega_{n-1}\alpha /2$ and considering the sequence $\rho_i$ defined above, one gets in \eqref{eq:prima} with the substitution $r=\rho_i$
	\[
	\left|D_{\chi_E} (B_{\rho_i}(x)) \cdot v -\omega_{n-1}\rho_i^{n-1}\right| = |A(x,\rho_i)| \leq \frac{\alpha}{2} \omega_{n-1}\rho_i^{n-1}\,.
	\]
	On the other hand, by \eqref{eq:step2}, we also have
	\[
	\left| D_{\chi_E}(B_{\rho_i}(x)) - D_{\chi_E}(B_{r_i}(x))\right| \leq P(E; B_{\rho_i}(x)\Delta B_{r_i}(x))\leq \omega_{n-1} |\rho_i^{n-1} -r_i^{n-1}| + o(r_i^{n-1}) = o(r_i^{n-1})
	\]
	as $i\to \infty$. Combining these two latter inequalities yields to
	\[
	\left| D_{\chi_E}(B_{r_i}(x))\cdot v - \omega_{n-1}r_i^{n-1}\right| \leq \frac{\alpha}{2}\omega_{n-1}\rho_i^{n-1} + o(r_i^{n-1}) =  \frac{\alpha}{2}\omega_{n-1}r_i^{n-1} + o(r_i^{n-1})\,,
	\]
	which contradicts \eqref{eq:contr} for $i$ large enough.
\end{proof}

\section{Existence Theorems}
This section is devoted to the proof of existence of solutions to the prescribed mean curvature equation \eqref{PMC}, that we recall here:
\begin{equation*}
\div Tu(x) = H(x),\qquad x\in \Omega\,.
\end{equation*}

In what follows we show that the weak regularity assumption, i.e. the validity of \eqref{eq:key} and \eqref{eq:weakmildreg}, coupled with the necessary condition \eqref{eq:necessarycondition} is enough to ensure existence of solutions to \eqref{PMC}. 

We will follow the argument of \cite{Giaquinta1974, Giusti1978}, which is based on the minimization of the functional
\begin{equation}\label{eq:functional}
\cJ[u] = \int_\Omega \sqrt{1+ |\grad u|^2}\ dx + \int_\Omega Hu\ dx +\int_{\partial \Omega} |u-\phi| \ d\H^{n-1},
\end{equation}
defined on $BV(\Omega)$, for a given $\phi\in L^{1}(\de\Omega)$. Note that the Euler-Lagrange equation of $\cJ$, obtained by perturbations with compact support in $\Omega$, is precisely equation \eqref{PMC}.  By Theorem \ref{thm:Mazya2011}, the last term in \eqref{eq:functional} is well-defined.

In the existence proof we will have first to discuss the easier \textit{non-extremal case}, in which the necessary condition \eqref{eq:necessarycondition} holds for the domain $\Omega$ as well, and then the more involved \textit{extremal case}, that is when \eqref{eq:ext} is satisfied.

First we need some preliminary results. The first one shows how to extend the necessary condition \eqref{eq:necessarycondition} to all measurable $A\subset \Omega$ such that $0<|A|<|\Omega|$. 
\begin{prop}\label{prop:subsets}
Let $\Omega\subset \R^{n}$ be a domain satisfying condition \eqref{eq:weakmildreg}. Assume that the necessary condition \eqref{eq:necessarycondition} holds for every $A\subset\subset \Omega$, then it also holds for every $A\subset \Omega$ such that $0<|A|<|\Omega|$. \end{prop}
\begin{proof}
Let us fix a measurable set $A\subset \Omega$ with $0<|A|<|\Omega|$ and finite perimeter. By Theorem \ref{thm:interiorapprox} there exists a sequence $\{\Omega_{j}\}_{j\in\N}$ of relatively compact, smooth open subsets of $\Omega$, such that $|\Omega \setminus \Omega_{j}|\to 0$ and $P(\Omega_{j}) \to P(\Omega)$ as $j\to\infty$. Now take $A_{j} = A\cap \Omega_{j}$ and notice that $A_{j}\subset\subset \Omega$, $P(A_{j})<+\infty$, and $A_{j}\to A$ in $L^{1}$ as $j\to\infty$. Since 
\[
P(A_{j}) + P(A\cup \Omega_{j}) \le P(A) + P(\Omega_{j}),
\] 
and owing to the fact that $A\cup \Omega_{j}\to \Omega$ in $L^{1}$ as $j\to\infty$, we deduce that
\begin{align*}
P(A) &\le \liminf_{j} P(A_{j}) \le \limsup_{j} P(A_{j}) \le \limsup_{j} \Big( P(A) + P(\Omega_{j}) - P(A\cup \Omega_{j})\Big)\\
& = P(A) + P(\Omega) - \liminf_{j} P(A\cup \Omega_{j}) \le P(A) + P(\Omega) - P(\Omega) = P(A), 
\end{align*}
which proves that 
\begin{equation}\label{PAjtoPA}
\lim_{j}P(A_{j}) = P(A).
\end{equation}
Now we observe that $P(A;\Omega)>0$, which follows from the connectedness of $\Omega$ coupled with the fact that $0<|A|<|\Omega|$. Therefore owing to \eqref{proprSchmidt} we can assume that $P(A_{j};\Omega_{j_{0}}) \ge c>0$ for a suitably large $j_{0}$ and for all $j\ge j_{0}$, which means that 
\begin{align*}
\left|\int_{A_{j}}H\, dx\right| &= \left|\int_{\de^{*} A_{j}} \langle Tu,\nu\rangle\, d\H^{n-1}\right| \le P(A_{j};\R^{n}\setminus \Omega_{j_{0}}) + \int_{\de^{*} A_{j}\cap \Omega_{j_{0}}} |\langle Tu,\nu\rangle |\, d\H^{n-1}\\ 
&\le P(A_{j};\R^{n}\setminus \Omega_{j_{0}}) + \alpha P(A_{j};\Omega_{j_{0}}) = P(A_{j}) - (1-\alpha)c,
\end{align*}

where $\alpha<1$ is the supremum of $|\langle Tu,\nu\rangle |$ on $\Omega_{j_{0}}$. Since $|A_{j}| \to |A|$ as $j\to \infty$, by the necessary condition written for $A_{j}$, and passing to the limit as $j\to\infty$, we get by \eqref{PAjtoPA}
\begin{equation}\label{eq:disnonopt}
\left|\int_{A} H\, dx\right| \le P(A) - (1-\alpha)c < P(A),
\end{equation}
whence the conclusion follows.
\end{proof}

The next lemma corresponds to \cite[Lemma 1.1]{Giusti1978}, thus we omit its proof.
\begin{lem}\label{lemmagiusti1.1}
Let $\Omega$ be a domain such that $\left|\int_{A}H\, dx\right| < P(A)$ holds for all $A\subset\Omega$ with the property that $|A|>0$. Then there exists $\e_{0}>0$ such that the stronger inequality
\[
\left|\int_{A}H\, dx\right| \le (1-\e_{0})P(A)
\]
holds for all such $A$.
\end{lem}

\begin{thm}[Existence, non-extremal case]\label{thm:subextremalexistence}
Let $\Omega$ be a weakly regular domain. If the necessary condition \eqref{eq:necessarycondition} holds also for $\Omega$, that is, we have the non-extremal condition \eqref{eq:nonext}, that is
\begin{equation*}
\left|\int_{\Omega} H\, dx\right| < P(\Omega),
\end{equation*} 
then the functional $\cJ$ defined in ~(\ref{eq:functional}) is minimized in $BV(\Omega)$.
\end{thm}

\begin{proof}
Fix a ball $B$ containing $\Omega$ and extend the function $H$ to $0$ in $B\setminus \Omega$. Fix a function $\Phi \in W^{1,1}_0(B)$ such that $\Phi = \phi$ on $\partial \Omega$ (this can be done according to Theorem \ref{thm:Mazya2011}). Then minimizing $\cJ$ on $BV(\Omega)$ is equivalent to minimizing $\cJt$ defined as
\[
\cJt: u \mapsto  \int_B \sqrt{1+ |\grad u|^2}\ dx + \int_B Hu\ dx,
\]
in $K = \{ u\in BV(B) | u = \Phi \quad \text{in $B\setminus \Omega$}\}$, which is a closed subset of $BV(B)$. Owing to Proposition \ref{prop:subsets} and by the assumption on $\Omega$ we can apply Lemma \ref{lemmagiusti1.1} and get the lower bound
\[
\int_\Omega Hu\ dx \geq -(1-\epsilon_0)\int_B |Du| - c\int_{\partial \Omega} |\phi|\ d\H^{n-1}
\]
for some $\e_{0}>0$, whence
\begin{equation}\label{eq:coercivity}
\cJt[u] \geq \epsilon_0 \int_B |Du|\ dx - c\int_{\partial \Omega} |\phi|\ d\H^{n-1}.
\end{equation}
Exploiting Poincar\'e's inequality on the ball $B$ one finally shows the coercivity of $\cJt$ in $L^1(\Omega)$. Since it is also lower semi-continuous within respect to the $L^{1}$-norm we infer the existence of a minimizer of $\cJt$ in $K$, hence of a minimizer of $\cJ$ in $BV(\Omega)$.
\end{proof}

In order to prove the existence of minimizers in the extremal case ~(\ref{eq:ext}), following \cite{Miranda1977} we introduce the notion of \textit{generalized solution} of \eqref{PMC}. For technical reasons, we consider the epigraph of $u$ instead of its subgraph, therefore the definition is slighty offset from the one in \cite{Miranda1977} (but of course equivalent up to changing the minus sign in \eqref{eq:epifunctional}).

\begin{defin}
A function $u : \Omega \to [-\infty, +\infty]$ is said to be a generalized solution to \eqref{PMC} if the epigraph of $u$
\[
U = \{(x,y) \in \Omega \times \mathbb{R}:\ y>u(x)\},
\]
minimizes the functional
\begin{equation}\label{eq:epifunctional}
P(U) - \int_{U} H \ dx\, dy,
\end{equation}
locally in $\Omega \times \mathbb{R}$.
\end{defin}

It is clear that any classical solution to \eqref{PMC} is also a generalized solution. Moreover, any generalized solution of \eqref{PMC} can be shown to satisfy some key properties, that we collect in the following proposition (see \cite{Giusti1978} and \cite{Miranda1964, Miranda1977} for the proof).

\begin{prop}\label{prop:Miranda}
Let $u$ be a generalized solution of \eqref{PMC} and define $N_{\pm} = \{x\in \Omega:\ u(x) = \pm\infty\}$. Then the following properties hold.
\begin{itemize}
\item[(i)] If $x\in N_{\pm}$ then $|N_{\pm}\cap B_{r}(x)|>0$ for all $r>0$.
\item[(ii)] The set $N_{\pm}$ minimizes the functional 
\[
E \mapsto P(E) \pm \int_{E} H\, dx
\]
locally in $\Omega$.
\item[(iii)] The function $u$ is smooth on $\Omega\setminus(N_{+}\cup N_{-})$.
\item[(iv)] Given a sequence $\{u_{k}\}$ of generalized solutions of \eqref{PMC}, then up to subsequences the epigraphs $U_{k}$ of $u_{k}$ converge to an epigraph $U$ of a function $u$ locally in $L^{1}(\Omega\times \R)$, moreover $u$ is a generalized solution of \eqref{PMC}.
\item[(v)] If $u$ is locally bounded, then $u$ is a classical solution of \eqref{PMC}. 
\end{itemize}
\end{prop}

The next lemma is a straightforward adaptation of \cite[Lemma 1.2]{Giusti1978}. The proof is the same up to choosing a sequence $\{\Omega_{j}\}_{j}$ as provided by Theorem \ref{thm:interiorapprox} with $\e = 1/j$.
\begin{lem}\label{lem:Giusti1.2}
Let $\Omega$ and $H(x)$ be such that \eqref{eq:necessarycondition}, \eqref{eq:weakmildreg} and \eqref{eq:ext} hold. Let $E\subset\Omega$ be a set of finite perimeter minimizing the functional
\[
P(E) - \int_{E}H\, dx
\]
locally in $\Omega$. Then either $E=\emptyset$ or $E=\Omega$, up to null sets.
\end{lem}

We now come to the existence of solutions of \eqref{PMC} in the extremal case. 

\begin{thm}[Existence, extremal case]\label{thm:extremal}
Let $\Omega$ be a weakly regular domain. Assume that \eqref{eq:necessarycondition} is satisfied and that the extremal condition \eqref{eq:ext} holds. Then there exists a solution $u$ of \eqref{PMC}.
\end{thm}

\begin{proof}
By Theorem \ref{thm:interiorapprox} we find a sequence of smooth, connected sets $\Omega_j\subset\subset \Omega$, such that $|\Omega \setminus \Omega_j|\to 0$ and $P(\Omega_{j})\to P(\Omega)$ as $j\to +\infty$. Since \eqref{eq:necessarycondition} holds for any $A \subset \Omega_j$ (and in particular for $A = \Omega_j$), in virtue of Theorem \ref{thm:subextremalexistence} (existence in the non-extremal case) we find a minimizer $u_j\in BV(\Omega_j)$ of $\cJ$ restricted to $BV(\Omega_j)$, as every $\Omega_{j}$ satisfies \eqref{eq:key}. Setting 
\[
t_{j} = \inf\Big\{t:\ |\{x\in \Omega_{j}:\ u_{j}(x)\ge t\}|\le |\Omega_{j}|/2\Big\}
\]
we obtain
\[
\min(|\{x\in \Omega_{j}:\ u_{j}(x)\ge t_{j}\}|,|\{x\in \Omega_{j}:\ u_{j}(x)\le t_{j}\}|) \ \ge\ |\Omega_{j}|/2\ \ge\ |\Omega|/4
\]
for all $j$ large enough. Therefore, we can consider the sequence of vertically translated functions $\{u_{j}(x)-t_{j}\}_{j}$ defined for $x\in \Omega_{j}$, and relabel it as $\{u_{j}\}_{j}$, so that 
\begin{equation}\label{eq:misurasoprasotto}
\min\Big(|\{x\in \Omega_{j}:\ u_{j}(x)\ge 0\}|,\ |\{x\in \Omega_{j}:\ u_{j}(x)\le 0\}|\Big)\ \ge\ |\Omega|/4
\end{equation}
for all $j$ large enough. Then, by applying Proposition \ref{prop:Miranda} (iv) on $\Omega_{j_{0}}$ for any fixed $j_{0}\in \N$, and by a diagonal argument, we infer that $u_{j}$ locally converges up to subsequences to a generalized solution $u$ as $j\to\infty$, in the sense that the epigraph $U_{j}$ locally converges to the epigraph of $u$ in $L^{1}_{loc}(\Omega\times \R)$ as $j\to\infty$. Let us set $N_{\pm} = \{x\in \Omega:\ u(x) = \pm\infty\}$ as in Proposition \ref{prop:Miranda}. We claim that $N_{\pm}$ are both empty, which in turn implies by Proposition \ref{prop:Miranda} (v) that $u$ is a classical solution of \eqref{PMC}. Indeed by Proposition \ref{prop:Miranda} (ii) the set $N_{-}$ minimizes the functional $P(E) - \int_{E}H\, dx$ defined for $E\subset\Omega$, thus by Lemma \ref{lem:Giusti1.2} we have either $N_{-}=\emptyset$ or $N_{-}=\Omega$. Similarly, the set $\Omega\setminus N_{+}$ minimizes $P(E) - \int_{E}H\, dx$ (this follows from the fact that $N_{+}$ minimizes $P(E) + \int_{E}H\, dx$), hence either $N_{+}=\Omega$ or $N_{+}=\emptyset$. By \eqref{eq:misurasoprasotto} we conclude that $N_{\pm} = \emptyset$, which proves our claim.
\end{proof}

\section{Characterization of extremality}
\label{section:extremality}
We have seen in the previous section that, given a domain $\Omega$ and a prescribed mean curvature function $H$, the condition \eqref{eq:necessarycondition} is necessary and sufficient for the existence of solutions to \eqref{PMC}, however the proof of this fact is different depending on the validity or not of the \textit{extremality condition} \eqref{eq:ext} (compare Theorems \ref{thm:subextremalexistence} and \ref{thm:extremal}). While in the non-extremal case the existence of solutions is genuinely variational, in the extremal case one recovers a solution as a limit of variational solutions defined on subdomains. Since extremality arises in physical models of capillarity for perfectly wetting fluids, the uniqueness and the stability of solutions with respect to suitable perturbations of the domain are of special interest in this case. 

In \cite{Giusti1976} Giusti showed that, assuming $C^{2}$ regularity of $\de\Omega$ and \eqref{eq:necessarycondition}, the extremality condition \eqref{eq:ext} is equivalent to a series of facts, and in particular to the uniqueness of the solution of \eqref{PMC} up to vertical translations.

Here we obtain essentially the same result only assuming that $\Omega$ is weakly regular. Before stating our main result, we present a list of properties using the same labels as those appearing in \cite{Giusti1978}. 

\begin{itemize}
\item[{\bf (E)}] \textit{(Extremality)} The pair $(\Omega,H)$ satisfies \eqref{eq:ext}, i.e., $\left|\int_{\Omega}H\, dx\right| = P(\Omega)$.
\item[{\bf (U)}] \textit{(Uniqueness)} The solution of \eqref{PMC} is unique up to vertical translations.
\item[{\bf (M)}] \textit{(Maximality)} $\Omega$ is maximal, i.e. no solution of  \eqref{PMC} can exist in any domain strictly containing $\Omega$.
\item[{\bf (V)}] \textit{(weak Verticality)} There exists a solution $u$ of \eqref{PMC} which is weakly vertical at $\partial \Omega$, i.e. 
\begin{equation*}
[Tu \cdot \nu] = 1\qquad \text{$\H^{n-1}$-a.e. on }\de\Omega\,,
\end{equation*}
where $[Tu \cdot \nu]$ is the weak normal trace of $Tu$ on $\de\Omega$. 
\item[{\bf (V')}] \textit{(integral Verticality)} There exists a solution $u$ of \eqref{PMC} and a sequence $\{\Omega_{i}\}_{i}$ of smooth subdomains, such that $\Omega_{i}\subset\subset \Omega$, $|\Omega\setminus \Omega_{i}|\to 0$, $P(\Omega_{i})\to P(\Omega)$, and 
\begin{equation*}
\lim_{i\to\infty}\int_{\partial \Omega_i} Tu (x) \cdot \nu \ d\H^{n-1} = P(\Omega),
\end{equation*}
as $i\to\infty$.
\end{itemize}

Then we come to the main result of this section.
\begin{thm}\label{thm:ACE}
Let $\Omega$ and $H$ be given, such that $\Omega$ is weakly regular and \eqref{eq:necessarycondition} holds. Then the properties (E), (U), (M), (V) and (V') are equivalent.
\end{thm}

Before proving Theorem \ref{thm:ACE} some further comments about properties (V) and (V') above are in order. In \cite{Giusti1978} the property (V) is stated in the stronger, pointwise form $Tu(x) = \nu(x)$ for all $x\in \de\Omega$ (moreover $\de\Omega$ is assumed of class $C^{2}$, hence $Tu$ can be continuously extended on $\de\Omega$ owing to well-known regularity results, see \cite{Emmer1976}) while (V') is stated by using the one-parameter family of inner parallel sets (which is again well-defined owing to the $C^{2}$-smoothness of $\de\Omega$).

The Maximum Principle Lemma that we state hereafter has been originally proved in \cite{Finn1965} and then in \cite{Giusti1978}. We remark that it remains valid under the weaker assumptions guaranteeing the interior smooth approximation property, in the sense of Theorem \ref{thm:interiorapprox}.
\begin{lem}[Maximum Principle]\label{lem:maxprinciple}
Let $\Omega\subset \R^{n}$ be open, bounded, connected and weakly regular. Let $u$ and $v$ be two functions of class $C^{2}(\Omega)$, such that $\div(Tu) \leq \div(Tv)$ in $\Omega$. Assume that $\partial \Omega = \Gamma_1 \cup \Gamma_2$ with $\Gamma_1$ relatively open in $\partial \Omega$, and $u,v \in C^0(\Omega \cup \Gamma_1)$ with $u\ge v$ on $\Gamma_1$. Assume further that
\[
\lim_{i\to \infty} \int_{\partial \Omega_i \setminus A} (1- Tu\cdot \nu)\ d\H^{n-1} = 0
\]
for every open set $A\supset \Gamma_1$, where $\{ \Omega_i \}_{i\in \N}$ is a sequence of smooth and relatively compact open subsets of $\Omega$, such that $|\Omega\setminus \Omega_{i}| \to 0$ and $P(\Omega_{i})\to P(\Omega)$ as $i\to\infty$. Then
\begin{itemize}
\item[(a)] if $\Gamma_1 \neq \emptyset$ then $u\geq v$ in $\Omega$;
\item[(b)] if $\Gamma_1 = \emptyset$ then $u= v+c$.
\end{itemize}
\end{lem}
\begin{proof}
In order to prove case (a) we first assume that $u>v$ on $\Gamma_1$. By the Gauss-Green formula on $\Omega_i$, for any positive function $\phi\in W^{1,\infty}(\Omega_i)$ one obtains
\begin{align*}
\int_{\Omega_i} (Tu - Tv)\cdot \nabla \phi &= - \int_{\Omega_i}\phi(\div Tu - \div Tv) + \int_{\de\Omega_i} \phi (Tu - Tv)\cdot \nu \\
& \ge \int_{\de\Omega_i} \phi (Tu - Tv)\cdot \nu \ge \int_{\de\Omega_i} \phi (Tu\cdot \nu -1)\,.
\end{align*}
Fix a positive constant $M>0$ and define the function $\phi_M(x) = \max\Big(0,\min(v-u,M)\Big)$. Of course $\phi_M \in W^{1,\infty}(\Omega_i)$ for all $i$ and $0\le \phi_M\le M$. Moreover, we can find an open set $A$ containing $\Gamma_1$ and such that $\phi_M = 0$ on $A\cap \Omega$. We also notice that
\[
(Tu - Tv)\cdot \nabla \phi_M = \begin{cases}
(Tu - Tv)\cdot (\nabla v - \nabla u) & \text{if } 0<v-u < M,\\
0 & \text{elsewhere}\,,
\end{cases}
\]
hence by a straightforward computation
\[
(Tu - Tv)\cdot \nabla \phi_M \le (|\nabla v| - |\nabla u|) \left(\frac{|\nabla u|}{\sqrt{1 + |\nabla u|^2}} - \frac{|\nabla v|}{\sqrt{1 + |\nabla v|^2}}\right) \le 0\,.
\]
Consequently, we obtain
\begin{align*}
\int_{\de\Omega_i\setminus A} \phi_M (Tu\cdot \nu -1) \le \int_{\Omega_i\setminus A} (Tu - Tv)\cdot \nabla \phi_M \le 0\,, 
\end{align*}
thus by taking the limit as $i\to\infty$ we find
\[
\int_{\Omega} (Tu - Tv)\cdot \nabla \phi_M = 0
\]
for all $M>0$. Therefore, setting $\phi = \max(v-u,0)$ we find $\nabla \phi = 0$ on $\Omega$, which means that $\phi$ is constant on $\Omega$. However, since $\phi = 0$ on $A\cap \Omega$ we deduce that $\phi = 0$, hence that $u\ge v$, on the whole $\Omega$. The full proof of case (a) is then completed by considering $v_\e = v+\e$ in place of $v$ and then letting $\e\to 0^+$.

Finally, for the proof of case (b) we fix $x_0\in \Omega$ and assume $v(x_0) = u(x_0) + 1$ up to a vertical translation. Arguing exactly as in the proof of case (a), 
we end up with $\phi$ constant on $\Omega$, where as before we set $\phi =\max(v-u,0)$. Since $\phi = \phi(x_0)=1$ we conclude that $v = u+1$ on $\Omega$, as wanted.

\end{proof}

We finally come to the proof of Theorem \ref{thm:ACE}. 
\begin{proof}[Proof of Theorem \ref{thm:ACE}] 
We shall split the proof in five steps.

\noindent
\textit{Step one: (E) $\Rightarrow$ (V').} Owing to (E) we have
\begin{align*}
P(\Omega) \stackrel{(E)}{=} \int_\Omega H\ dx = \lim_{i\to \infty} \int_{\Omega_i} H\ dx = \lim_{i\to \infty} \int_{\Omega_i} \div(Tu)\ dx = \lim_{i\to \infty} \int_{\partial \Omega_i} Tu (x) \cdot \nu \ d\H^{n-1}\,, 
\end{align*}
which implies (V'). 
\medskip

\noindent
\textit{Step two: (E) $\Leftrightarrow$ (M).} 
Let us start by showing (E)$\implies$(M). We argue by contradiction and suppose there exists a solution $u$ of \eqref{PMC} defined on $\widetilde{\Omega} \supsetneq \Omega$. Then Proposition \ref{prop:subsets} gives
\[
\left| \int_\Omega H\ dx \right | <  P(\Omega),
\]
which immediately contradicts (E). Let us now show the implication (M)$\implies$(E). Again by contradiction we assume that 
\[
\left|\int_{A}H\, dx\right| < P(A)
\]
for all $A\subset \Omega$. By Lemma \ref{lemmagiusti1.1} there exists $\e_{0}>0$ such that 
\begin{equation}\label{unomeneps1}
\left|\int_{A}H\, dx\right| < (1-\e_{0})P(A)
\end{equation}
for all $A\subset \Omega$. Now we claim that (compare with Lemma 2.1 in \cite{Giusti1978}) given a ball $B$ such that $\Omega\subset\subset B$, for all $0<\e<\e_{0}$ one can find an open set $\Omega_{\e}\subset B$ with smooth boundary, such that $\Omega\subset\subset \Omega_{\e}$ and 
\begin{equation}\label{unomeneps2}
\left|\int_{A}H\, dx\right| < (1-\e)P(A),\qquad \forall\, A\subset \Omega_{\e}.
\end{equation}
Of course, the validity of \eqref{unomeneps2} would allow us to apply Theorem \ref{thm:subextremalexistence} on $\Omega_{\e}$, which in turn would contradict our assumption (M). In order to show \eqref{unomeneps2} we argue again by contradiction, i.e., we assume that there exists $\e\in (0,\e_{0})$ such that, for every $U$ with smooth boundary satisfying $\Omega\subset\subset U$, one can find $A\subset U$ for which \eqref{unomeneps2} fails. In particular, for every $k\in \N$ we may choose a suitable $U_{k}$ as specified below, such that $\Omega\subset\subset U_{k}$, $|U_{k}\setminus \Omega|<1/k$, $\de U_{k}$ is smooth and there exists $A_{k}\subset U_{k}$ for which
\begin{equation}\label{unomeneps3}
\left|\int_{A_{k}}H\, dx\right| \ge (1-\e)P(A_{k})
\end{equation}
holds. By \eqref{unomeneps3} we have that 
\[
P(A_{k})\le \frac{|B|\sup_{B}|H|}{1-\e}\qquad \forall\, k\in \N\,,
\]
hence we can extract a not relabeled subsequence $A_{k}$ converging to some $A\subset B$ in $L^{1}$. On the other hand, since $|A_{k}\setminus \Omega|\le |U_{k}\setminus \Omega|\to 0$ as $k\to\infty$, we infer that $A\subset \Omega$ up to null sets. By \eqref{unomeneps3}, by the lower semi-continuity of the perimeter and by the continuity of the term $\int_{A_{k}}H\, dx$ with respect to $L^{1}$-convergence, we conclude that
\[
\left|\int_{A}H\, dx\right| \ge (1-\e)P(A)
\]
which is in contrast with \eqref{unomeneps1}. We are left to prove that such a sequence $U_{k}$ exists. To this aim we consider the open set $V = B\setminus \overline{\Omega}$ and notice that $P(V) = P(B) + P(\Omega) = \H^{n-1}(\de B)+\H^{n-1}(\de \Omega) = \H^{n-1}(\de V)$ owing to the assumption on $\Omega$. We can now apply Theorem \ref{thm:interiorapprox} to $V$ with $\delta_{k} = \min(\dist(\de B,\de\Omega)/3,1/k)$ and set $U_{k} = B\setminus (V_{\delta_{k}}\cup {\mathcal N}_{2\delta_{k}}(\de B))$. Thanks to \eqref{proprSchmidt} we find that $\de U_{k}$ is smooth, $\Omega\subset \subset U_{k}$ and $|U_{k}\setminus \Omega|<\delta_{k}\le 1/k$, as wanted. 
\medskip

\noindent
\textit{Step three: (V') $\implies$ (U).} We consider two solutions $u,v$ of \eqref{PMC}, then if we take $\Gamma_1 = \emptyset$ and thanks to the property $P(\Omega_{i})\to P(\Omega)$ as $i\to\infty$, we infer that the assumptions of Lemma \ref{lem:maxprinciple}(b) are satisfied. Consequently there exists a constant $c\in\R$ such that $u = v+c$. 
\medskip

\noindent
\textit{Step four: (U) $\implies$ (E).} Let $u$ be the unique solution of $\div(Tu) = H$ on $\Omega$, up to vertical translations. By contradiction we suppose that 
\[
\int_\Omega H \ dx < P(\Omega)\,.
\]
Arguing as in Step two we find a bounded and smooth domain $\tilde{\Omega} \supsetneq \Omega$ for which \eqref{eq:necessarycondition} holds. By Theorems \ref{thm:subextremalexistence} and \ref{thm:extremal} there exists a solution $\tilde{u}$ of $\div (T\tilde{u}) = H$ on $\tilde{\Omega}$. Then (U) implies the existence of $t\in \mathbb{R}$ such that $u=\tilde{u}+t$ on $\Omega$. By internal regularity of $\tilde{u}$, we infer that $u\in C^{1}(\overline\Omega)$. Fix now a function $\phi \in C^2(\R^{n})$ such that
\begin{equation}\label{eq:utile}
\H^{n-1} (\{ x\in \partial \Omega : \ \phi(x) - u(x) \neq s\}) >0\qquad \forall\, s\in \R.
\end{equation}
The choice of $\phi$ satisfying \eqref{eq:utile} can be easily made as follows: if $u$ is constant on $\de \Omega$, then one can choose any smooth function $\phi$ taking different values on two distinct points of $\de\Omega$; conversely, if $u$ is not constant on $\de\Omega$ then one can take $\phi=0$. 
Now we consider a minimizer $w$ of the functional
\[
\int_\Omega \sqrt{1 + |Dw|^2} + \int_\Omega Hw + \int_{\partial \Omega} |w - \phi|\ d\H^{n-1},
\]
then $w$ necessarily satisfies \eqref{PMC}. By the assumed uniqueness up to translations one has that $w = u+s$ for some $s\in \R$. Then it follows that 
\begin{equation}\label{tusottouno}
|Tu(x_{0})| = |Tw(x_{0})| <1. 
\end{equation}
Moreover by \eqref{eq:utile} we have that $w \neq \phi$ on some set $K\subset \partial^* \Omega$ with $\H^{n-1}(K)>0$. Fix now a point $x_0 \in K$ and assume without loss of generality that $\phi(x_0) > w(x_0)$. Set now $\C = \Omega\times \R$, $p_{0} = (x_{0},w(x_{0}))\in \de \C$, and notice that by the continuity of $w$ and $\phi$ on $\partial \Omega$ there exists $R>0$ such that the subgraph of $\phi$ contains the ball $B_R(p_0)\subset \R^{n+1}$. Owing to the choice of $B_{R}(p_{0})$, the epigraph
\[
W := \{ p=(x,y)\in \C:\ y>w(x)\} 
\]
necessarily minimizes the functional
\[
P(W;B_{R}(p_{0})) - \int_{W\cap B_{R}(p_{0})} H 
\]
with obstacle $ \R^{n+1}\setminus \C$ inside $B_{R}(p_{0})$. In other words, for any set $U$ that coincides with $W$ outside the set $A:=B_R(p_0) \cap \overline{\C}$, one has
\begin{equation}\label{eq:minimalitadiW}
P(W;B_{R}(p_{0})) - \int_{W\cap B_{R}(p_{0})} H \leq P(U;B_{R}(p_{0})) - \int_{U\cap B_{R}(p_{0})} H .
\end{equation}
It is then easy to show that $W$ is a $(\Lambda, R)$-perimeter minimizer in $\overline \C$ (see Definition \ref{def:LambdaMinimizer}), where $R$ is the radius of the ball defined above and $\Lambda = \sup_{\Omega} |H|$. Indeed for any ball $B_r \subset B_{R}(p_{0})$ and any set $U$ such that $U\Delta W \subset \subset B_r \cap \overline{\C}$, by \eqref{eq:minimalitadiW} one has that 
\begin{align*}
P(W; B_r) &= P(W;B_{R}(p_{0})) - P(W;B_{R}(p_{0}) \setminus \overline{B_{r}})\\\nonumber
&\leq P(U;B_{R}(p_{0}))- P(U;B_{R}(p_{0}) \setminus \overline{B_{r}}) - \int_{B_{R}(p_{0})} H(\chi_U - \chi_W) \nonumber \\
& \leq P(U; B_r) + \sup_{\Omega}|H|\, |U\Delta W|,
\end{align*}
which proves the $(\Lambda, R)$-minimality of $W$ in $\overline{\C}$. Then by Theorem \ref{thm:lambdamin} we infer that $\nu_{W}(p_{0}) = \nu_{\C}(p_{0})$, which contradicts \eqref{tusottouno}.
\medskip

\noindent
\textit{Step five: (V) and (V') are equivalent.} We can consider the sequence $\Omega_{j}$ of Theorem \ref{thm:interiorapprox} and apply Theorem \ref{thm:GaussGreen} to get
\begin{equation}\label{EDequiv}
\int_{\Omega\setminus\Omega_{j}} H(x)\, dx = \int_{\Omega\setminus\Omega_{j}} \div Tu(x)\, dx = \int_{\de\Omega}[Tu\cdot\nu]\, d\H^{n-1} - \int_{\de\Omega_{j}} Tu \cdot \nu_{j}\, d\H^{n-1}\,.
\end{equation}
Now, observing that the left-hand side of \eqref{EDequiv} is infinitesimal as $j\to\infty$ the equivalence between (V) and (V') is immediate. 

The proof is finally completed by combining the previous five steps. 
\end{proof}

We now show a well-known consequence of Lemma \ref{lem:maxprinciple}, which can be obtained by arguing as in Step two of the proof of Theorem \ref{thm:ACE}. 
\begin{prop}\label{prop:uboundedbelow}
Assume that $u$ is a solution of \eqref{PMC} on $\Omega$ and that either (V) or (V') holds. Then $u$ is bounded from below.
\end{prop}
\begin{proof}
Let $B$ denote a ball compactly contained in $\Omega$ and consider the open set $S = \Omega \setminus \overline{B}$. By Lemma \ref{lemmagiusti1.1} and arguing as in Step two of the proof of Theorem \ref{thm:ACE} we find a solution $w$ of \eqref{PMC} which is of class $C^{1}(\overline S)$. Since in particular $u\in C^{2}(\overline B)$ we can assume that $w\le u$ on $\partial B$ up to a vertical translation, hence by Lemma \ref{lem:maxprinciple}(a) we deduce that $w\le u$ on $S$, which gives the conclusion at once.
\end{proof}

We conclude the section with some remarks about the stability of solutions of \eqref{PMC} in the extremal case. One might ask whether or not there exists some perturbation $(\Omega_\e, H_\e)$ of an extremal pair $(\Omega, H)$, such that $(\Omega_\e, H_\e)$ satisfies the necessary condition \eqref{eq:necessarycondition} and the solution $u_{\e}$ of \eqref{PMC} on $\Omega_{\e}$ is in a suitable sense a small perturbation of $u$ up to translations, as soon as $\e$ is small. The following proposition contains a result in this direction.
\begin{prop}[Stability]\label{prop:stability}
Let $\{\Omega_{j}\}_{j}$ be a sequence of bounded domains and $\{H_j\}_j$ a sequence of Lipschitz functions, such that $\Omega_{j}$ is weakly regular and the pair $(\Omega_j, H_j)$ is extremal. Assume moreover that $\Omega_{j}\to \Omega_{\infty}$ in $L^{1}$ and $P(\Omega_{j})\to P(\Omega_{\infty})$, as $j\to\infty$, with $\Omega_{\infty}$ weakly regular, and that $H_{j}$ uniformly converges to $H_\infty$ such that the pair $(\Omega_\infty, H_\infty)$ is extremal as well. Then the sequence of unique (up to translations) solutions $\{u_j\}_j$ to the \eqref{PMC} problem for the pair $(\Omega_j, H_j)$ converges to a solution $u_{\infty}$ of \eqref{PMC} for the pair $(\Omega_{\infty},H_{\infty})$, in the sense of the $L^{1}_{loc}$-convergence of the epigraphs.
\end{prop}
\begin{proof}
Due to our hypotheses, the existence of a solution $u_{j}$ to \eqref{PMC} for the pair $(\Omega_{j},H_{j})$ (also for $j=\infty$) is guaranteed by Theorem \ref{thm:extremal}. Arguing as in Theorem \ref{thm:extremal}, for any $j$ large enough we can find a suitable $t_j$ such that the translated solution $u_j + t_j$ which we just rename $u_j$ satisfies
\[
\min\Big(|\{x\in \Omega_{j}:\ u_{j}(x)\ge 0\}|,\ |\{x\in \Omega_{j}:\ u_{j}(x)\le 0\}|\Big)\ \ge\ |\Omega|/4\,.
\]
Then we find that the epigraphs $U_j$ of $u_j$ converge in $L^{1}_{loc}(\R^{n+1})$ to a set $U_{\infty}^{*}$ which is the epigraph of a classical solution $u_{\infty}^{*}$ defined on $\Omega_{\infty}$. 
By Theorem \ref{thm:ACE} we have that $u_{\infty}^{*} = u_{\infty}$ up to a translation, thus the thesis follows. 
\end{proof}

In the recent paper \cite{LeoSar2-2016}, an explicit example of an extremal pair $(\Omega, H)$ and of a sequence of extremal pairs $(\Omega_j, H_j)$ satisfying the hypotheses of Proposition \ref{prop:stability} is constructed, for the special case $H_j = P(\Omega_j)/|\Omega_j|$, by removing a sequence of smaller and smaller disks from the unit disk in $\R^{2}$, in such a way that it looks like a sort of Swiss cheese with holes accumulating towards a portion of its boundary (see Figure \ref{fig:poroso} and Example \ref{ex:porosita} below; for a more complete discussion we refer to \cite{LeoSar2-2016}). 

This shows the following, remarkable fact: while a generic small and smooth perturbation of the unit disk may produce a dramatic change in the capillary solution (and even end up with non-existence of a solution), there exist some non-smooth perturbations that, instead, preserve both existence and stability.

\begin{ex}\label{ex:porosita}
\begin{figure}[t]
\centering
\includegraphics[trim={1.5cm 5cm 1.5cm 4.5cm},clip, width=\textwidth]{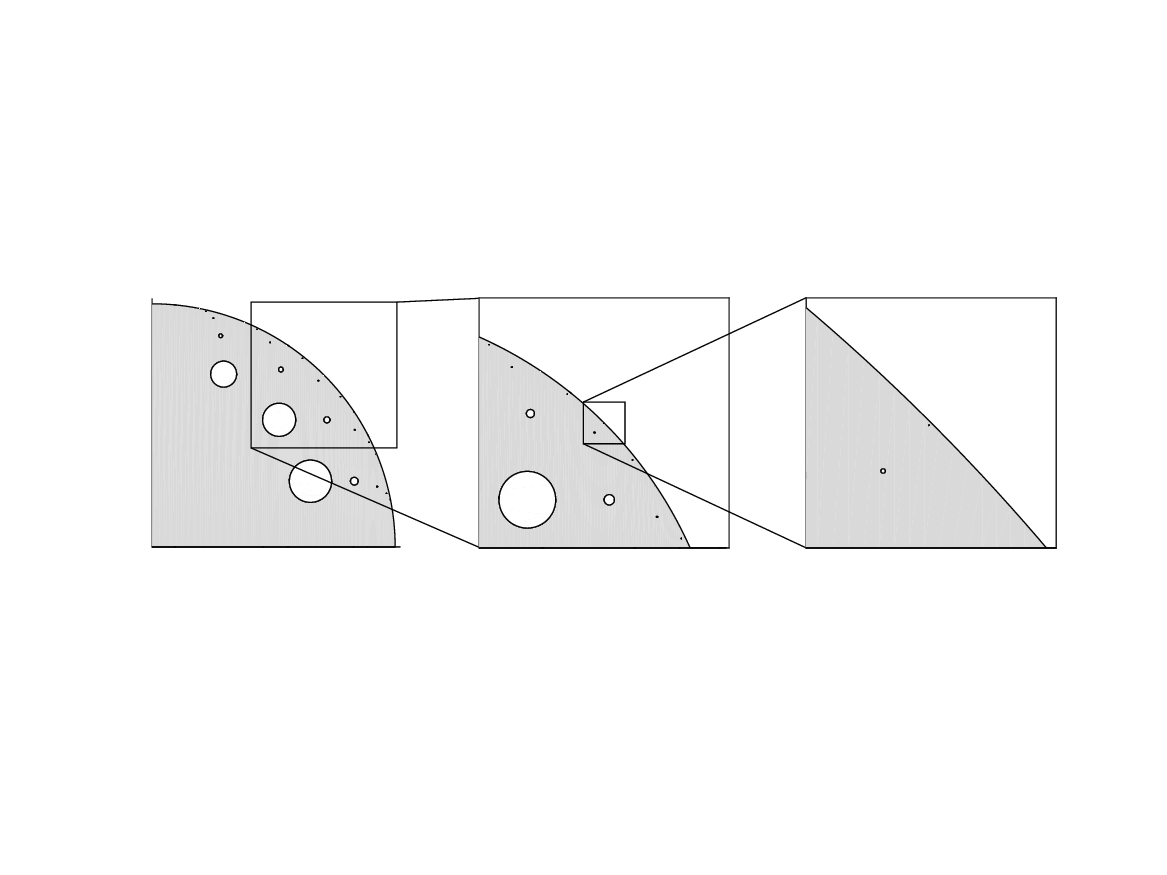} 
\caption{The ``Swiss cheese'' set $O_{a,\delta,\e}$ constructed in Example \ref{ex:porosita}
}
\label{fig:poroso}
\end{figure}
Let $0<\delta <\epsilon < 1$ and $a>1$ be fixed. For $i\geq 1$ and $j=1,\dots, i$ we set
\begin{equation*}
\rho_{ij} = 1- \frac{\epsilon}{a^{i^2+j}}, \qquad r_{ij} = \frac{\delta}{a^{2i^2+2j}}, \qquad \theta_{ij} = \frac{\pi}{2}\frac{j}{i+1},
\end{equation*}
Then we define
\[
O_{a,\delta,\e} = B_1 \setminus \bigcup_{i,j}\overline{B_{ij}},
\]
where $B_{1}\subset \R^{2}$ is the unit disk centered at the origin, and 
\[
B_{ij} := B_{r_{ij}} ((\rho_{ij}\cos(\theta_{ij}), \rho_{ij}\sin(\theta_{ij}))
\] 
(see Figure \ref{fig:poroso}). 
We prove in \cite{LeoSar2-2016} that for a suitable choice of parameters $a,\delta,\e$ the open set $O_{a,\delta,\e}$ fulfils the hypotheses of Theorem \ref{thm:ACE}.

One can then build a sequence of non-smooth perturbations of the unit disk by simply filling one hole of the Swiss cheese at a time: indeed this operation creates a sequence of  subdomains of the unit disk that satisfy the hypotheses of Proposition \ref{prop:stability} with the choice $H_j = P(\Omega_j)/|\Omega_j|$ and with $(B_{1},2)$ as the limit extremal pair.
\end{ex}

%\section*{Conflict of Interest}
%The authors declare that they have no conflict of interest.

\bibliographystyle{plain}

%\bibliography{cap201503}

\end{document}